\documentclass[]{amsart}
\usepackage{amsmath}
\usepackage{amsfonts}
\usepackage{amssymb}
\usepackage{amsthm}
\usepackage{enumerate}
\usepackage{enumitem}
\usepackage{graphicx,color,xcolor,tikz}
\usepackage{bbold}
\usepackage{tikz-cd}
\usepackage[utf8]{inputenc}
\usepackage{hyperref}
\hypersetup{
    colorlinks=true,                         
    linkcolor=blue,
    citecolor=red, 
  } 

\newtheorem{Lem}{Lemma}
\newtheorem{Thm}{Theorem}
\newtheorem{Prop}{Proposition}

\newtheorem{Rem}{Remark}

\renewcommand{\dfrac}[2]{\frac{#1}{#2}}

\newcommand{\dt}{\Delta t}
\newcommand{\dx}{\Delta x}
\newcommand{\sig}{\sigma} % CFL number
\title{Relative Entropy Analysis of the Jin-Xin Model: Theory and Numerics}

\date{\today}

 \author{C. Mahmoud}
 \address{Institut Montpelliérain Alexander Grothendieck, Université de
   Montpellier, CNRS, Montpellier, France}
 \email{christina.mahmoud@umontpellier.fr}

\begin{document}

\begin{abstract}
This paper studies the relaxation limit and the numerical approximation of the linear Jin-Xin model using the relative entropy method. Under the subcharacteristic condition, we establish a first-order relative entropy estimate with respect to the relaxation parameter towards the equilibrium transport equation. We then analyze two fully discrete asymptotic-preserving methods: a Lie splitting scheme and a staggered scheme. For each method, the relaxation solution is compared with the solution of the corresponding limiting discrete scheme. We derive discrete relative entropy identities and obtain global estimates that are uniform with respect to the relaxation parameter for well-prepared data. Numerical experiments at fixed mesh and final time reveal an additional quadratic branch for small values of the relaxation parameter, which is not detected by the global estimates. An exact fixed-time analysis identifies the mechanism responsible for this sharper behavior and proves the quadratic rate observed when the relaxation scale is smaller than the numerical time step.
\end{abstract}

\maketitle


 \noindent
 \textbf{2020 MCS.} 35L40, 65M08, 65M12

\tableofcontents

\section{Introduction}
\label{sec-introduction}

Relaxation systems provide a useful framework for the analysis and numerical
approximation of conservation laws. In this approach, a target conservation
law is approximated by a hyperbolic relaxation system involving additional
variables and a stiff source term scaled by a small parameter
\(\varepsilon>0\), called the relaxation parameter. The source term drives the
solution towards an equilibrium manifold, and the target conservation law is
recovered in the relaxation limit \(\varepsilon\to0\).

The Jin-Xin model \cite{JinXin1995} is a standard prototype of this
framework. It provides an explicit description of the interplay between
hyperbolic propagation, relaxation towards equilibrium, and entropy
dissipation. Its stability is governed by the subcharacteristic condition
introduced in the general framework of \cite{ChenLevermoreLiu94}. This
condition relates the characteristic speeds of the relaxation system to those
of the equilibrium equation and ensures the existence of a strictly convex
entropy.

The relative entropy method provides a natural stability functional for
comparing a solution of the relaxation system with a reference solution of the
limiting equation; see, for example,
\cite{DafermosBook10,Tzavaras05,LattanzioTzavaras12}. At the fully discrete
level, the relative entropy analysis must account simultaneously for the
numerical viscosity of the convective flux and for the dissipation induced by
the discrete treatment of the stiff source term
\cite{BerthonBessemoulinMathis16,BessemoulinMathis24,
BulteauBerthonBessemoulin19}.

This issue is particularly relevant for asymptotic-preserving schemes. Such
schemes remain stable under a Courant-Friedrichs-Lewy restriction independent
of the relaxation parameter and reduce, as \(\varepsilon\to0\), to a
consistent discretization of the equilibrium equation
\cite{jin2010asymptotic,Jin2022,JinLevermore96}.

In this paper, we establish continuous and fully discrete relative entropy
estimates for the linear Jin-Xin model. At the discrete level, we study a Lie
splitting scheme and a staggered asymptotic-preserving scheme. For each method,
the relaxation solution is compared with the solution of the corresponding
limiting discrete scheme lifted onto the equilibrium manifold. For
well-prepared initial data, the resulting global estimates are uniform with
respect to the relaxation parameter and reproduce the continuous
\(O(\varepsilon)\) relative entropy bound.

The numerical experiments reveal a sharper behavior that is not identified
by these global estimates. As shown in
Figure~\ref{fig:intro-fixed-time}, the discrete relative entropy exhibits a
quadratic branch for the smallest values of the relaxation parameter. This
observation motivates the fixed-time analysis carried out in the final
section. There, we derive and iterate the discrete relations satisfied by the
difference between the relaxation solution and the corresponding limiting
discrete solution, in order to explain the observed
\(O(\varepsilon^2)\) regime.

\begin{figure}[!t]
  \centering
  \includegraphics[width=0.88\textwidth]{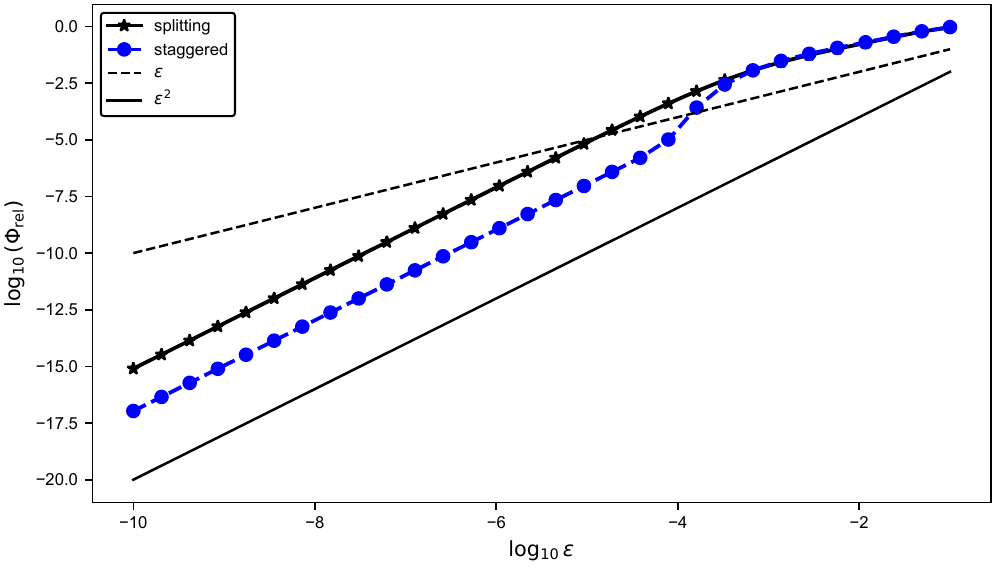}
  \caption{Fixed-time discrete relative entropy for the Lie splitting and
  staggered schemes. The dashed and solid reference lines have slopes \(1\)
  and \(2\), respectively}
  \label{fig:intro-fixed-time}
\end{figure}

\noindent\textbf{Organization of the paper.}
Section~\ref{sec:continuous-setting} presents the continuous Jin-Xin model and
the corresponding relative entropy estimate.
Section~\ref{sec:numerical-schemes} analyzes the Lie splitting scheme, its
limiting discretization, and the associated discrete entropy structure.
Section~\ref{sec:disc-rel-entropy-stag} treats the staggered scheme and its
discrete relative entropy estimate. Finally,
Section~\ref{sec:fixed-time-regimes} presents the numerical results and
explains the quadratic branch through a fixed-time error analysis.
%=======================================================================================
\section{Continuous framework and model presentation}
\label{sec:continuous-setting}

%==============================================
% Linear Jin-Xin model
%==============================================
\subsection{The linear Jin-Xin model}

We consider the one-dimensional linear transport equation
\begin{equation}
  \partial_t u + a\,\partial_x u = 0,
  \qquad
  (t,x)\in\mathbb{R}_+\times\mathbb{R},
  \label{eq:scalar-conservation}
\end{equation}
where \(u=u(t,x)\) is the scalar unknown and
\(a\in\mathbb{R}\) is the constant transport velocity.\\
A Jin-Xin relaxation approximation of
\eqref{eq:scalar-conservation}
\cite{ChenLevermoreLiu94,JinXin1995,SerreRelax2000}
is given by
\begin{equation*}
  \begin{cases}
    \partial_t u+\partial_x v=0,\\[0.3em]
    \partial_t v+\lambda^2\partial_x u
    =
    \dfrac{1}{\varepsilon}(au-v),
  \end{cases}
  \qquad
  (t,x)\in\mathbb{R}_+\times\mathbb{R},
  \tag{$H_\varepsilon$}
  \label{eq:jin-xin-system}
\end{equation*}
where \(v=v(t,x)\) is the relaxation variable,
\(\lambda>0\) is the characteristic speed of the homogeneous system, and
\(\varepsilon>0\) is the relaxation parameter. The source term becomes stiff
as \(\varepsilon\to0\) and drives the solution towards equilibrium.\\
Introducing
\[
  \mathbf U
  :=
  \begin{pmatrix}
    u\\
    v
  \end{pmatrix},
  \qquad
  A
  :=
  \begin{pmatrix}
    0&1\\
    \lambda^2&0
  \end{pmatrix},
  \qquad
  \mathbf R(\mathbf U)
  :=
  \begin{pmatrix}
    0\\
    au-v
  \end{pmatrix},
\]
system \eqref{eq:jin-xin-system} can be written in vector form as
\begin{equation}
  \partial_t\mathbf U
  +
  \partial_x\mathbf F(\mathbf U)
  =
  \dfrac{1}{\varepsilon}\mathbf R(\mathbf U),
  \qquad
  \mathbf F(\mathbf U)=A\mathbf U.
  \label{eq:JX-vector}
\end{equation}
The eigenvalues of \(A\) are \(-\lambda\) and \(\lambda\). Hence, since
\(\lambda>0\), the homogeneous part of the relaxation system is strictly
hyperbolic.\\
Formally, in the relaxation limit \(\varepsilon\to0\), the source term imposes
\begin{equation}\label{eq:local-equilibrium}
  \mathbf R(\bar{\mathbf U})=0 \Longleftrightarrow  \bar v=a\bar u.
\end{equation}
The equilibrium manifold is therefore
\begin{equation}
  \mathcal M_{\mathrm{eq}}
  :=
  \left\{
    (\bar u,\bar v)\in\mathbb{R}^2
    \,;\,
    \bar v=a\bar u
  \right\}.
  \label{eq:equilibrium-manifold}
\end{equation}
Substituting the equilibrium relation
\eqref{eq:local-equilibrium} into the first equation of
\eqref{eq:jin-xin-system} gives the limiting transport equation
\begin{equation}
  \partial_t\bar u+a\,\partial_x\bar u=0.
  \label{eq:limiting-transport}
\end{equation}
To compare the relaxation system with its limit, we lift the limiting solution
onto the equilibrium manifold by setting
\[
  \bar{\mathbf U}
  :=
  \begin{pmatrix}
    \bar u\\
    a\bar u
  \end{pmatrix}.
\]
Using \eqref{eq:limiting-transport}, the lifted equilibrium solution satisfies
\begin{equation*}
  \begin{cases}
    \partial_t\bar u+\partial_x\bar v=0,\\[1mm]
    \partial_t\bar v+\lambda^2\partial_x\bar u
    =
    (\lambda^2-a^2)\partial_x\bar u.
  \end{cases}
  \tag{$H_0$}
  \label{eq:equilibre}
\end{equation*}
The stability of the relaxation limit is governed by the strict
subcharacteristic condition
\begin{equation}
  |a|<\lambda.
  \label{eq:subchar-final}
\end{equation}
It requires the characteristic speed \(a\) of the equilibrium equation to lie
strictly between the two characteristic speeds \(-\lambda\) and \(\lambda\)
of the relaxation system.
%==============================================
% Entropic structure
%==============================================
\subsection{Entropic structure of the Jin-Xin model}
\label{sec:physical-entropy}

We associate with system \eqref{eq:jin-xin-system} the quadratic entropy
\begin{equation}
  \eta(\mathbf U)
  :=
  \frac12\mathbf U^\top H\mathbf U
  =
  \frac{\lambda^2}{2}u^2
  +
  \frac12v^2
  -
  auv,
  \label{eq:eta-def}
\end{equation}
where
\begin{equation}
  H
  :=
  \begin{pmatrix}
    \lambda^2&-a\\
    -a&1
  \end{pmatrix}.
  \label{eq:H-def}
\end{equation}
Under the subcharacteristic condition
\eqref{eq:subchar-final}, the matrix \(H\) is positive definite. Therefore,
\(\eta\) is strictly convex. The associated entropy flux is
\begin{equation}
  \Psi(\mathbf U)
  :=
  \frac12\mathbf U^\top HA\mathbf U
  =
  \lambda^2uv
  -
  \frac{a\lambda^2}{2}u^2
  -
  \frac a2v^2.
  \label{eq:q-def}
\end{equation}
Since \(HA\) is symmetric, \((\eta,\Psi)\) defines an entropy-entropy flux
pair for the homogeneous part of the Jin-Xin system. Multiplying \eqref{eq:JX-vector} by
\(\nabla\eta(\mathbf U)=H\mathbf U\) gives the entropy balance
\begin{equation}
  \partial_t\eta(\mathbf U)
  +
  \partial_x\Psi(\mathbf U)
  =
  -\dfrac{1}{\varepsilon}(v-au)^2
  \leq0.
  \label{eq:entropy-balance}
\end{equation}
Thus, the relaxation source dissipates entropy and vanishes precisely on the
equilibrium manifold.

%==============================================
% Entropy at equilibrium
%==============================================
\subsection{Entropy behavior at equilibrium}
\label{sec:entropy-equilibrium}

For an equilibrium state
\[
  \bar{\mathbf U}
  =
  (\bar u,a\bar u)^\top,
\]
the entropy and entropy flux defined in
\eqref{eq:eta-def} and \eqref{eq:q-def} reduce to
\begin{equation}
\begin{aligned}
  \bar\eta(\bar u)
  &:=
  \eta(\bar{\mathbf U})
  =
  \frac{\lambda^2-a^2}{2}\bar u^2,\\[0.3em]
  \bar\Psi(\bar u)
  &:=
  \Psi(\bar{\mathbf U})
  =
  \frac{a(\lambda^2-a^2)}{2}\bar u^2.
\end{aligned}
\label{eq:entropie-equilibre}
\end{equation}
Under \eqref{eq:subchar-final}, the reduced entropy
\(\bar\eta\) is strictly convex.
Moreover, if \(\bar u\) satisfies the limiting transport equation
\eqref{eq:limiting-transport}, then
\begin{equation}
  \partial_t\eta(\bar{\mathbf U})
  +
  \partial_x\Psi(\bar{\mathbf U})
  =
  0.
  \label{eq:entropy-balance-Ubar}
\end{equation}
Finally, since \(\eta\) and \(\Psi\) are quadratic, for every
\(\delta\mathbf U\in\mathbb{R}^2\) one has
\begin{equation}
\begin{aligned}
  \eta(\mathbf U+\delta\mathbf U)-\eta(\mathbf U)
  &=
  \nabla\eta(\mathbf U)\cdot\delta\mathbf U
  +
  \frac12
  \delta\mathbf U^\top H\delta\mathbf U,\\
  \Psi(\mathbf U+\delta\mathbf U)-\Psi(\mathbf U)
  &=
  \nabla\eta(\mathbf U)\cdot A\delta\mathbf U
  +
  \frac12
  \delta\mathbf U^\top HA\delta\mathbf U.
\end{aligned}
\label{eq:eta-identity-local-app}
\end{equation}
These identities will be used repeatedly in the continuous and discrete
relative entropy analyses.
%==========================================================
% Relative entropy
%==========================================================
\subsection{Relative entropy}
\label{sec:relative-entropy}

The relative entropy method is a classical stability tool for conservation
laws endowed with a convex entropy; see, for instance,
\cite{DafermosBook10,LattanzioTzavaras12,SerreBook99}. In the present
relaxation setting, it allows us to compare a solution of the Jin-Xin system
with a reference solution of the limiting transport equation lifted onto the
equilibrium manifold.

Let \(\mathbf U\) be a sufficiently regular solution of
\eqref{eq:jin-xin-system} on \([0,T]\times\mathbb R\). Let \(\bar u\) be a
sufficiently regular solution of \eqref{eq:scalar-conservation}, and define
\[
  \bar v:=a\bar u,
  \qquad
  \bar{\mathbf U}:=
  \begin{pmatrix}
    \bar u\\
    \bar v
  \end{pmatrix}.
\]
By construction, \(\bar{\mathbf U}\) takes values in the equilibrium manifold
\(\mathcal M_{\mathrm{eq}}\) defined in
\eqref{eq:equilibrium-manifold}.\\
The relative entropy and the associated relative entropy flux are defined by
\begin{equation}
\label{eq:relative-entropy-def}
\begin{aligned}
  \eta_{\mathrm{rel}}
  \bigl(\mathbf U\mid\bar{\mathbf U}\bigr)
  &:=
  \eta(\mathbf U)
  -
  \eta(\bar{\mathbf U})
  -
  \nabla\eta(\bar{\mathbf U})
  \cdot
  \bigl(\mathbf U-\bar{\mathbf U}\bigr),\\[0.4em]
  \Psi_{\mathrm{rel}}
  \bigl(\mathbf U\mid\bar{\mathbf U}\bigr)
  &:=
  \Psi(\mathbf U)
  -
  \Psi(\bar{\mathbf U})
  -
  \nabla\eta(\bar{\mathbf U})
  \cdot
  \bigl(
    \mathbf F(\mathbf U)-\mathbf F(\bar{\mathbf U})
  \bigr).
\end{aligned}
\end{equation}
Since the Jin-Xin system is linear and the entropy and entropy flux are
quadratic, identities \eqref{eq:eta-identity-local-app} give
\begin{equation}
\label{eq:relative-entropy-quadratic}
\begin{aligned}
  \eta_{\mathrm{rel}}
  \bigl(\mathbf U\mid\bar{\mathbf U}\bigr)
  &=
  \frac12
  \bigl(\mathbf U-\bar{\mathbf U}\bigr)^\top
  H
  \bigl(\mathbf U-\bar{\mathbf U}\bigr),\\[0.4em]
  \Psi_{\mathrm{rel}}
  \bigl(\mathbf U\mid\bar{\mathbf U}\bigr)
  &=
  \frac12
  \bigl(\mathbf U-\bar{\mathbf U}\bigr)^\top
  HA
  \bigl(\mathbf U-\bar{\mathbf U}\bigr).
\end{aligned}
\end{equation}
Under the subcharacteristic condition
\eqref{eq:subchar-final}, the matrix \(H\) is positive definite. Therefore,
\[
  \eta_{\mathrm{rel}}
  \bigl(\mathbf U\mid\bar{\mathbf U}\bigr)\geq0,
\]
with equality if and only if
\(\mathbf U=\bar{\mathbf U}\).

We define the relative entropy functional by
\begin{equation}
\label{eq:Phi-continuous}
  \Phi(t)
  :=
  \int_{\mathbb R}
  \eta_{\mathrm{rel}}
  \bigl(
    \mathbf U(t,x)\mid\bar{\mathbf U}(t,x)
  \bigr)\,dx.
\end{equation}
Since \(H\) is positive definite under the subcharacteristic condition
\eqref{eq:subchar-final}, the relative entropy functional is equivalent to the
squared \(L^2\)-distance between the relaxation solution and the equilibrium
solution. More precisely, there exist constants \(c_1,c_2>0\), depending only
on \(a\) and \(\lambda\), such that
\begin{equation}
\label{eq:Phi-L2-equivalence-short}
  c_1
  \left\|
    \mathbf U(t,\cdot)-\bar{\mathbf U}(t,\cdot)
  \right\|_{L^2(\mathbb R)}^2
  \leq
  \Phi(t)
  \leq
  c_2
  \left\|
    \mathbf U(t,\cdot)-\bar{\mathbf U}(t,\cdot)
  \right\|_{L^2(\mathbb R)}^2
\end{equation}
for every \(t\in[0,T]\).
%----------------------------------------------------------
\subsection{Local relative entropy identity}

We first derive the local balance satisfied by the relative entropy.

\begin{Prop}[Local relative entropy identity]
\label{prop:local-relative-entropy-JX}
Assume the subcharacteristic condition \eqref{eq:subchar-final}. Let
\(\mathbf U=(u,v)^\top\) be a smooth solution of
\eqref{eq:jin-xin-system}, and let \(\bar u\) be a smooth solution of
\eqref{eq:scalar-conservation}. Set
\[
  \bar v=a\bar u,
  \qquad
  \bar{\mathbf U}=(\bar u,\bar v)^\top.
\]
Then
\begin{equation}
\label{eq:rel-entropy-balance}
\begin{aligned}
  \partial_t
  \eta_{\mathrm{rel}}
  \bigl(\mathbf U\mid\bar{\mathbf U}\bigr)
  +
  \partial_x
  \Psi_{\mathrm{rel}}
  \bigl(\mathbf U\mid\bar{\mathbf U}\bigr)
  &=
  -\frac{1}{\varepsilon}
  \bigl[
    a(u-\bar u)-(v-\bar v)
  \bigr]^2\\
  &\quad
  +
  (\lambda^2-a^2)
  \bigl[
    a(u-\bar u)-(v-\bar v)
  \bigr]
  \partial_x\bar u.
\end{aligned}
\end{equation}
\end{Prop}

\begin{proof}
Set
\[
  \Delta\mathbf U:=\mathbf U-\bar{\mathbf U}.
\]
By \eqref{eq:relative-entropy-quadratic} and the symmetry of \(H\) and \(HA\),
\[
  \partial_t\eta_{\mathrm{rel}}
  +
  \partial_x\Psi_{\mathrm{rel}}
  =
  (\Delta\mathbf U)^\top H
  \left(
    \partial_t\Delta\mathbf U
    +
    A\partial_x\Delta\mathbf U
  \right).
\]
Subtracting the embedded equilibrium system \eqref{eq:equilibre} from
\eqref{eq:JX-vector} gives
\[
  \partial_t\Delta\mathbf U
  +
  A\partial_x\Delta\mathbf U
  =
  \begin{pmatrix}
    0\\[0.2em]
    \varepsilon^{-1}(au-v)
    -
    (\lambda^2-a^2)\partial_x\bar u
  \end{pmatrix}.
\]
Since
\[
  \Delta\mathbf U
  =
  \begin{pmatrix}
    u-\bar u\\
    v-\bar v
  \end{pmatrix},
\]
a direct multiplication by \(H\) gives
\[
  H\Delta\mathbf U
  =
  \begin{pmatrix}
    \lambda^2(u-\bar u)-a(v-\bar v)\\[0.3em]
    -a(u-\bar u)+(v-\bar v)
  \end{pmatrix}.
\]
Since the right-hand side of the equation for
\(\Delta\mathbf U\) has a vanishing first component, only the second component
of \(H\Delta\mathbf U\) contributes to the scalar product. Moreover, using
\(\bar v=a\bar u\), we have
\[
  au-v
  =
  a(u-\bar u)-(v-\bar v).
\]
Consequently,
\[
\begin{aligned}
  \partial_t\eta_{\mathrm{rel}}
  +
  \partial_x\Psi_{\mathrm{rel}}
  &=
  \bigl[-a(u-\bar u)+(v-\bar v)\bigr]
  \left[
    \frac{1}{\varepsilon}
    \bigl(a(u-\bar u)-(v-\bar v)\bigr)
    -
    (\lambda^2-a^2)\partial_x\bar u
  \right],
\end{aligned}
\]
which gives \eqref{eq:rel-entropy-balance}.
\end{proof}

%----------------------------------------------------------
\subsection{Global relative entropy estimate}

We now integrate the local identity to obtain the continuous stability
estimate.

\begin{Prop}[Global relative entropy estimate]
\label{prop:relative-entropy-JX}
Let \(\mathbf U=(u,v)^\top\) be a smooth solution of
\eqref{eq:jin-xin-system} on \([0,T]\times\mathbb R\), and let \(\bar u\) be
a smooth solution of \eqref{eq:scalar-conservation}. Set
\[
  \bar v=a\bar u,
  \qquad
  \bar{\mathbf U}=(\bar u,\bar v)^\top.
\]
Assume that:

\begin{enumerate}
\item[(i)]
the strict subcharacteristic condition \eqref{eq:subchar-final} holds;

\item[(ii)]
for every \(t\in[0,T]\), the relaxation and equilibrium solutions have the
same constant far-field states:
\begin{equation}
\label{eq:states-at-infinity}
  \lim_{x\to-\infty}\mathbf U(t,x)
  =
  \lim_{x\to-\infty}\bar{\mathbf U}(t,x)
  =
  \mathbf U_-,
  \qquad
  \lim_{x\to+\infty}\mathbf U(t,x)
  =
  \lim_{x\to+\infty}\bar{\mathbf U}(t,x)
  =
  \mathbf U_+,
\end{equation}
where \(\mathbf U_-\) and \(\mathbf U_+\) belong to
\(\mathcal M_{\mathrm{eq}}\);

\item[(iii)]
there exists \(K>0\) such that
\begin{equation}
\label{eq:dubar-L2-bound}
  \int_0^t
  \int_{\mathbb R}
  |\partial_x\bar u(s,x)|^2\,dx\,ds
  \leq K^2,
  \qquad
  t\in[0,T].
\end{equation}
\end{enumerate}

Then, for every \(t\in[0,T]\),
\begin{equation}
\label{eq:relative-entropy-full}
\begin{aligned}
  \Phi(t)
  &+
  \frac{1}{2\varepsilon}
  \int_0^t
  \int_{\mathbb R}
  \bigl[
    a(u-\bar u)-(v-\bar v)
  \bigr]^2
  \,dx\,ds\\
  &\leq
  \Phi(0)
  +
  \frac{\varepsilon(\lambda^2-a^2)^2}{2}
  \int_0^t
  \int_{\mathbb R}
  |\partial_x\bar u(s,x)|^2
  \,dx\,ds.
\end{aligned}
\end{equation}
In particular,
\begin{equation}
\label{eq:relative-entropy-coarse-K}
  \Phi(t)
  \leq
  \Phi(0)
  +
  \frac{\varepsilon(\lambda^2-a^2)^2}{2}K^2.
\end{equation}
\end{Prop}

\begin{proof}
Integrating \eqref{eq:rel-entropy-balance} over
\((0,t)\times\mathbb R\), the relative entropy flux contribution vanishes by
the far-field condition \eqref{eq:states-at-infinity}. Hence,
\[
\begin{aligned}
  \Phi(t)-\Phi(0)
  &=
  -\frac{1}{\varepsilon}
  \int_0^t
  \int_{\mathbb R}
  \bigl[
    a(u-\bar u)-(v-\bar v)
  \bigr]^2
  \,dx\,ds\\
  &\quad
  +
  (\lambda^2-a^2)
  \int_0^t
  \int_{\mathbb R}
  \bigl[
    a(u-\bar u)-(v-\bar v)
  \bigr]
  \partial_x\bar u
  \,dx\,ds.
\end{aligned}
\]
By Young's inequality,
\[
\begin{aligned}
  &(\lambda^2-a^2)
  \bigl[
    a(u-\bar u)-(v-\bar v)
  \bigr]
  \partial_x\bar u\\
  &\qquad\leq
  \frac{1}{2\varepsilon}
  \bigl[
    a(u-\bar u)-(v-\bar v)
  \bigr]^2
  +
  \frac{\varepsilon(\lambda^2-a^2)^2}{2}
  |\partial_x\bar u|^2.
\end{aligned}
\]
Substituting this inequality into the previous identity and integrating gives
\eqref{eq:relative-entropy-full}. Estimate
\eqref{eq:relative-entropy-coarse-K} then follows from
\eqref{eq:dubar-L2-bound}.
\end{proof}

Combining \eqref{eq:relative-entropy-coarse-K} with the equivalence
\eqref{eq:Phi-L2-equivalence-short}, we obtain
\begin{equation}
\label{eq:continuous-L2-estimate}
  \sup_{t\in[0,T]}
  \left\|
    \mathbf U(t,\cdot)-\bar{\mathbf U}(t,\cdot)
  \right\|_{L^2(\mathbb R)}^2
  \leq
  C\bigl(\Phi(0)+\varepsilon\bigr),
\end{equation}
where \(C>0\) depends only on \(a\), \(\lambda\), and \(K\).

In particular, if the initial data are well prepared in the sense that
\begin{equation}
\label{eq:continuous-well-prepared-data}
  \Phi(0)=0,
\end{equation}
then
\[
  \sup_{t\in[0,T]}
  \left\|
    \mathbf U(t,\cdot)-\bar{\mathbf U}(t,\cdot)
  \right\|_{L^2(\mathbb R)}
  =
  O(\varepsilon^{1/2}).
\]
Consequently,
\[
  \mathbf U
  \longrightarrow
  \bar{\mathbf U}
  \qquad
  \text{in }
  L^\infty\bigl(0,T;L^2(\mathbb R)\bigr)
  \qquad
  \text{as }\varepsilon\to0.
\]

The continuous analysis therefore yields an \(O(\varepsilon)\) relative entropy estimate and quantifies the entropy dissipation induced by the relaxation source. We now investigate how this structure is preserved at the fully discrete level.
%%%%%%%%%%%%%%%%%%%%%%%%%%%%%%%%%%%%%%%%%%%%%%%%%%%%%%%%%%%%%%%%%%
%==========================================================
% Numerical schemes
%==========================================================
\section{Lie splitting scheme and discrete entropy structure}
\label{sec:numerical-schemes}

%----------------------------------------------------------
\subsection{Definition of the Lie splitting scheme}
\label{subsec:jx-splitting-scheme}

We consider a uniform spatial grid \(x_j=j\dx\), \(j\in\mathbb Z\), with mesh
size \(\dx>0\), and the discrete times
\[
  t^n=n\dt,
  \qquad
  n\in\mathbb N,
\]
where \(\dt>0\) is the time step. We denote by
\[
  \mathbf U_j^n
  =
  \begin{pmatrix}
    u_j^n\\[2pt]
    v_j^n
  \end{pmatrix}
\]
the discrete approximation of \(\mathbf U\) at position \(x_j\) and time
\(t^n\), and set
\[
  \sig:=\frac{\dt}{\dx}.
\]

The Lie splitting scheme separates the homogeneous hyperbolic part from the
relaxation source term
\cite{JinLevermore96,Natalini1996}. First, the homogeneous system is advanced
by the conservative step
\begin{equation}
\label{eq:conv-step}
  \mathbf U_j^{n+\frac12}
  =
  \mathbf U_j^n
  -
  \sig
  \left(
    \widehat{\mathbf F}_{j+\frac12}^n
    -
    \widehat{\mathbf F}_{j-\frac12}^n
  \right).
\end{equation}
For the linear Jin-Xin system, we use the Rusanov numerical flux
\begin{equation}
\label{eq:HLL-flux}
  \widehat{\mathbf F}_{j+\frac12}^n
  =
  \frac12
  \left(
    \mathbf F(\mathbf U_j^n)
    +
    \mathbf F(\mathbf U_{j+1}^n)
  \right)
  -
  \frac{\lambda}{2}
  \left(
    \mathbf U_{j+1}^n-\mathbf U_j^n
  \right).
\end{equation}
Since the characteristic speeds of the homogeneous system are
\(-\lambda\) and \(\lambda\), this flux coincides with the corresponding
Harten-Lax-van Leer flux
\cite{HartenLaxvanLeer1983}.

The relaxation source term is then treated implicitly:
\begin{equation}
\label{eq:source-step}
  \mathbf U_j^{n+1}
  =
  \mathbf U_j^{n+\frac12}
  +
  \frac{\dt}{\varepsilon}
  \mathbf R(\mathbf U_j^{n+1}).
\end{equation}
The time step satisfies the Courant-Friedrichs-Lewy (CFL) condition
\begin{equation}
\label{CFL}
  \lambda\sig\leq1.
\end{equation}
Equations \eqref{eq:conv-step} and \eqref{eq:source-step} define the fully
discrete Lie splitting approximation of the Jin-Xin system.

%----------------------------------------------------------
\subsection{Discrete entropy identity}
\label{sec:disc-entropy-scheme}

We now derive the discrete counterpart of the continuous entropy balance
\eqref{eq:entropy-balance}. For every \(j\in\mathbb Z\) and
\(n\in\mathbb N\), we set
\[
  \eta_j^n:=\eta(\mathbf U_j^n),
  \qquad
  \eta_j^{n+\frac12}:=\eta(\mathbf U_j^{n+\frac12}),
\]
and define
\begin{equation}
\label{eq:DeltaF-def}
  \Delta\mathbf F_j^n
  :=
  \widehat{\mathbf F}_{j+\frac12}^n
  -
  \widehat{\mathbf F}_{j-\frac12}^n.
\end{equation}

\begin{Thm}[Discrete entropy identity over one time step]
\label{thm:full-disc-entropy}
Let
\((\mathbf U_j^n)_{j\in\mathbb Z,n\in\mathbb N}\)
be the sequence generated by
\eqref{eq:conv-step}-\eqref{eq:source-step}. Define the numerical entropy flux
by
\begin{equation}
\label{eq:qhat-matrix}
  \Psi_{j+\frac12}^n
  :=
  \frac14
  \left(
    \nabla\eta(\mathbf U_j^n)
    \cdot\mathbf F(\mathbf U_{j+1}^n)
    +
    \nabla\eta(\mathbf U_{j+1}^n)
    \cdot\mathbf F(\mathbf U_j^n)
  \right)
  =
  \frac12
  (\mathbf U_j^n)^\top HA\mathbf U_{j+1}^n.
\end{equation}
For the Jin-Xin system, this flux is explicitly given by
\begin{equation}
\label{eq:qhat-explicit}
  \Psi_{j+\frac12}^n
  =
  \frac{\lambda^2}{2}
  \left(
    u_j^n v_{j+1}^n
    +
    u_{j+1}^n v_j^n
  \right)
  -
  \frac{a\lambda^2}{2}u_j^n u_{j+1}^n
  -
  \frac a2v_j^n v_{j+1}^n.
\end{equation}

We introduce
\begin{align}
  R_j^{1,n}
  &:=
  \frac{\lambda}{4}
  \Big[
    (\mathbf U_{j+1}^n-\mathbf U_j^n)^\top
    H(\mathbf U_{j+1}^n-\mathbf U_j^n)
    \nonumber\\[-0.2em]
  &\hspace{4.3cm}
    +
    (\mathbf U_j^n-\mathbf U_{j-1}^n)^\top
    H(\mathbf U_j^n-\mathbf U_{j-1}^n)
  \Big],
  \label{eq:Di-def}\\[0.4em]
  R_j^{2,n}
  &:=
  -\frac{\lambda}{4}
  \Big[
    (\mathbf U_{j+1}^n)^\top H\mathbf U_{j+1}^n
    -
    2(\mathbf U_j^n)^\top H\mathbf U_j^n
    +
    (\mathbf U_{j-1}^n)^\top H\mathbf U_{j-1}^n
  \Big],
  \label{eq:Ri-def}\\[0.4em]
  R_j^{3,n}
  &:=
  (\Delta\mathbf F_j^n)^\top
  H\Delta\mathbf F_j^n.
  \label{eq:Xi-def}
\end{align}
Then the following identity holds:
\begin{align}
\label{eq:full-disc-eta-identity}
  \eta_j^{n+1}
  +
  \sig
  \left(
    \Psi_{j+\frac12}^n-\Psi_{j-\frac12}^n
  \right)
  &=
  \eta_j^n
  -
  \sig R_j^{1,n}
  -
  \sig R_j^{2,n}
  +
  \frac{\sig^2}{2}R_j^{3,n}
  \nonumber\\
  &\quad
  -
  \frac{\dt}{\varepsilon}
  \bigl(v_j^{n+1}-au_j^{n+1}\bigr)^2
  -
  \frac{\dt^2}{2\varepsilon^2}
  \bigl(v_j^{n+1}-au_j^{n+1}\bigr)^2.
\end{align}
\end{Thm}

\begin{proof}
We first consider the convective step. By
\eqref{eq:conv-step} and \eqref{eq:DeltaF-def},
\[
  \mathbf U_j^{n+\frac12}
  =
  \mathbf U_j^n
  -
  \sig\Delta\mathbf F_j^n.
\]
Since the entropy is quadratic, identity
\eqref{eq:eta-identity-local-app} gives
\begin{equation}
\label{eq:eta-increment}
  \eta_j^{n+\frac12}-\eta_j^n
  =
  -\sig
  \nabla\eta(\mathbf U_j^n)
  \cdot\Delta\mathbf F_j^n
  +
  \frac{\sig^2}{2}
  (\Delta\mathbf F_j^n)^\top
  H\Delta\mathbf F_j^n.
\end{equation}
The last term in \eqref{eq:eta-increment} is
\(\frac{\sig^2}{2}R_j^{3,n}\).

Using the Rusanov flux \eqref{eq:HLL-flux}, we obtain
\begin{equation}
\label{eq:DeltaF-Rusanov}
  \Delta\mathbf F_j^n
  =
  \frac12
  \left(
    \mathbf F(\mathbf U_{j+1}^n)
    -
    \mathbf F(\mathbf U_{j-1}^n)
  \right)
  -
  \frac{\lambda}{2}
  \left(
    \mathbf U_{j+1}^n
    -
    2\mathbf U_j^n
    +
    \mathbf U_{j-1}^n
  \right).
\end{equation}
Since
\(\nabla\eta(\mathbf U)=H\mathbf U\),
\(\mathbf F(\mathbf U)=A\mathbf U\), and \(HA\) is symmetric, the first term
on the right-hand side of \eqref{eq:DeltaF-Rusanov} gives
\[
  \Psi_{j+\frac12}^n-\Psi_{j-\frac12}^n.
\]
Expanding the second term gives
\[
\begin{aligned}
  -\frac{\lambda}{2}
  (\mathbf U_j^n)^\top H
  \left(
    \mathbf U_{j+1}^n
    -
    2\mathbf U_j^n
    +
    \mathbf U_{j-1}^n
  \right)=
  R_j^{1,n}+R_j^{2,n}.
\end{aligned}
\]
Therefore,
\[
  \nabla\eta(\mathbf U_j^n)
  \cdot\Delta\mathbf F_j^n
  =
  \Psi_{j+\frac12}^n-\Psi_{j-\frac12}^n
  +
  R_j^{1,n}
  +
  R_j^{2,n}.
\]
Substituting this expression into \eqref{eq:eta-increment} yields
\begin{equation}
\label{eq:disc-eta-identity-local}
  \eta_j^{n+\frac12}
  +
  \sig
  \left(
    \Psi_{j+\frac12}^n-\Psi_{j-\frac12}^n
  \right)
  =
  \eta_j^n
  -
  \sig R_j^{1,n}
  -
  \sig R_j^{2,n}
  +
  \frac{\sig^2}{2}R_j^{3,n}.
\end{equation}

We now consider the implicit relaxation step. From
\eqref{eq:source-step},
\[
  \mathbf U_j^{n+\frac12}-\mathbf U_j^{n+1}
  =
  -\frac{\dt}{\varepsilon}
  \mathbf R(\mathbf U_j^{n+1}).
\]
Applying the quadratic identity \eqref{eq:eta-identity-local-app} at
\(\mathbf U_j^{n+1}\) gives
\[
\begin{aligned}
  \eta_j^{n+\frac12}
  =
  \eta_j^{n+1}
  +
  \frac{\dt}{\varepsilon}
  \bigl(v_j^{n+1}-au_j^{n+1}\bigr)^2
  +
  \frac{\dt^2}{2\varepsilon^2}
  \bigl(v_j^{n+1}-au_j^{n+1}\bigr)^2,
\end{aligned}
\]
where we used
\[
  \nabla\eta(\mathbf U)\cdot\mathbf R(\mathbf U)
  =
  -(v-au)^2,
  \qquad
  \mathbf R(\mathbf U)^\top H\mathbf R(\mathbf U)
  =
  (v-au)^2.
\]
Combining this identity with
\eqref{eq:disc-eta-identity-local} proves
\eqref{eq:full-disc-eta-identity}.
\end{proof}

\begin{Rem}
Under the subcharacteristic condition \eqref{eq:subchar-final}, the matrix
\(H\) is positive definite. Consequently,
\[
  R_j^{1,n}\geq0,
  \qquad
  R_j^{3,n}\geq0.
\]
The last two terms in \eqref{eq:full-disc-eta-identity} represent the entropy
dissipation induced by the implicit relaxation step.
\end{Rem}
%%%%%%%%%%%%%%%%%%%%%%%%%%%%%%%%%%%%%%%%%%%%
%%%%%%%%%%%%%%%%%%%%%%%%%%%%%%%%%%%%%%%%%%%%%%%%%%%%%%%%%%%%%%%%%%%%%
\subsection{Discrete entropy identity at equilibrium}
\label{sec:disc-entropy-equilibrium}

We now discretize the limiting transport equation
\eqref{eq:limiting-transport} and lift the resulting scalar approximation onto
the equilibrium manifold.
We consider the finite volume scheme
\begin{equation}
\label{eq:limit-scheme}
  \bar u_j^{n+1}
  =
  \bar u_j^n
  -
  \sig
  \left(
    \widehat f_{j+\frac12}^n
    -
    \widehat f_{j-\frac12}^n
  \right),
\end{equation}
where
\begin{equation}
\label{eq:rusanov-flux-scalar}
  \widehat f_{j+\frac12}^n
  :=
  \frac a2
  \left(
    \bar u_j^n+\bar u_{j+1}^n
  \right)
  -
  \frac{\lambda}{2}
  \left(
    \bar u_{j+1}^n-\bar u_j^n
  \right)
\end{equation}
is the Rusanov flux with numerical speed \(\lambda\).

The corresponding lifted equilibrium state is
\begin{equation}
\label{eq:disc-eq-manifold}
  \bar{\mathbf U}_j^n
  :=
  \begin{pmatrix}
    \bar u_j^n\\[2pt]
    a\bar u_j^n
  \end{pmatrix}
  \in\mathcal M_{\mathrm{eq}}.
\end{equation}
Thus, the discrete reference solution remains on the equilibrium manifold for
every \(j\in\mathbb Z\) and \(n\in\mathbb N\).

The scalar scheme
\eqref{eq:limit-scheme}-\eqref{eq:rusanov-flux-scalar} is obtained formally
from the Lie splitting scheme in the limit \(\varepsilon\to0\). The numerical
speed therefore remains equal to \(\lambda\), although the characteristic
speed of the limiting transport equation is \(a\). Accordingly, the CFL
condition remains \(\lambda\sig\leq1\).

On the equilibrium manifold, the entropy becomes
\begin{equation}
\label{eq:entropie-discrete-equilibre}
  \bar\eta_j^n
  :=
  \eta(\bar{\mathbf U}_j^n)
  =
  \frac{\lambda^2-a^2}{2}
  (\bar u_j^n)^2.
\end{equation}
The corresponding numerical entropy flux is obtained by evaluating
\eqref{eq:qhat-matrix} at the equilibrium states:
\begin{equation}
\label{eq:barPsi-explicit}
  \bar\Psi_{j+\frac12}^n
  :=
  \frac{a(\lambda^2-a^2)}{2}
  \bar u_j^n\bar u_{j+1}^n.
\end{equation}

We define
\begin{equation}
\label{eq:delta-f-equilibrium}
  \Delta\widehat f_j^n
  :=
  \widehat f_{j+\frac12}^n
  -
  \widehat f_{j-\frac12}^n.
\end{equation}
The corresponding lifted flux difference is
\begin{equation}
\label{eq:flux-equilibre}
  \Delta\bar{\mathbf F}_j^n
  :=
  \begin{pmatrix}
    \Delta\widehat f_j^n\\[2pt]
    a\Delta\widehat f_j^n
  \end{pmatrix}.
\end{equation}
Consequently,
\[
  \bar{\mathbf U}_j^{n+1}
  =
  \bar{\mathbf U}_j^n
  -
  \sig\Delta\bar{\mathbf F}_j^n.
\]

We introduce
\begin{align}
  \bar R_j^{1,n}
  &:=
  \frac{\lambda(\lambda^2-a^2)}{4}
  \left[
    \left(
      \bar u_{j+1}^n-\bar u_j^n
    \right)^2
    +
    \left(
      \bar u_j^n-\bar u_{j-1}^n
    \right)^2
  \right],
  \label{eq:barDi-def}\\[0.4em]
  \bar R_j^{2,n}
  &:=
  -\frac{\lambda(\lambda^2-a^2)}{4}
  \left[
    (\bar u_{j+1}^n)^2
    -
    2(\bar u_j^n)^2
    +
    (\bar u_{j-1}^n)^2
  \right],
  \label{eq:barRi-def}\\[0.4em]
  \bar R_j^{3,n}
  &:=
  (\lambda^2-a^2)
  \left(
    \Delta\widehat f_j^n
  \right)^2.
  \label{eq:barXi-def}
\end{align}

\begin{Prop}[Discrete entropy identity for the limiting scheme]
\label{prop:disc-entropy-equilibrium}
Let
\((\bar u_j^n)_{j\in\mathbb Z,n\in\mathbb N}\)
be the solution of \eqref{eq:limit-scheme}. Then
\begin{equation}
\label{eq:disc-eta-eq-identity}
\begin{aligned}
  \bar\eta_j^{n+1}
  +
  \sig
  \left(
    \bar\Psi_{j+\frac12}^n
    -
    \bar\Psi_{j-\frac12}^n
  \right)
  &=
  \bar\eta_j^n
  -
  \sig\bar R_j^{1,n}
  -
  \sig\bar R_j^{2,n}
  +
  \frac{\sig^2}{2}
  \bar R_j^{3,n}.
\end{aligned}
\end{equation}
\end{Prop}

\begin{proof}
Using \eqref{eq:limit-scheme} and the quadratic form of the reduced entropy,
we obtain
\[
\begin{aligned}
  \bar\eta_j^{n+1}-\bar\eta_j^n
  =
  -\sig
  (\lambda^2-a^2)
  \bar u_j^n
  \Delta\widehat f_j^n
  +
  \frac{\sig^2}{2}
  (\lambda^2-a^2)
  \left(
    \Delta\widehat f_j^n
  \right)^2.
\end{aligned}
\]
The last term is
\(\frac{\sig^2}{2}\bar R_j^{3,n}\).

Using the scalar Rusanov flux
\eqref{eq:rusanov-flux-scalar}, a direct calculation gives
\[
\begin{aligned}
  (\lambda^2-a^2)
  \bar u_j^n
  \Delta\widehat f_j^n
  &=
  \bar\Psi_{j+\frac12}^n
  -
  \bar\Psi_{j-\frac12}^n
  +
  \bar R_j^{1,n}
  +
  \bar R_j^{2,n}.
\end{aligned}
\]
Substituting this expression into the entropy increment gives
\eqref{eq:disc-eta-eq-identity}.
\end{proof}

Under the subcharacteristic condition
\eqref{eq:subchar-final},
\[
  \bar R_j^{1,n}\geq0,
  \qquad
  \bar R_j^{3,n}\geq0.
\]
%%%%%%%%%%%%%%%%%%%%%%%%%%%%%%%%%%%%%%%%%%%%%%%%%%%%%%%%%%%%%%%%%%%
%==========================================================
% Discrete relative entropy
%==========================================================
\subsection{Discrete relative entropy for the Lie splitting scheme}
\label{sec:disc-rel-entropy-JX}

We derive a fully discrete relative entropy inequality for the Lie splitting
scheme. Unlike the continuous analysis of Section~\ref{sec:relative-entropy},
both states in the comparison are discrete: the relaxation solution is
compared with the solution of the limiting scheme lifted onto the equilibrium
manifold.\\
For a scalar sequence \((\bar u_j^n)\), we define the discrete interface
derivative by
\begin{equation}
\label{eq:discrete-derivation}
  (\tilde\delta_x\bar u)_{j+\frac12}^n
  :=
  \frac{\bar u_{j+1}^n-\bar u_j^n}{\dx}.
\end{equation}
Under the subcharacteristic condition \eqref{eq:subchar-final}, the matrix
\(H\) is positive definite. We therefore use the associated norm
\begin{equation}
\label{eq:norme-H}
  \|\mathbf X\|_H^2
  :=
  \mathbf X^\top H\mathbf X,
  \qquad \mathbf X\in\mathbb R^2.
\end{equation}
%%%%%%%%%%%%%%%%%%%%%%%%%%%%%%%%%%%%%%%%%%%%%%%%%%%%%%%%%%%%%%%%%%%%%%%%%%%
For all \(j\in\mathbb Z\) and \(n\in\mathbb N\), we set
\begin{equation}
\label{eq:DeltaU-def-rel}
  \Delta\mathbf U_j^n
  :=
  \mathbf U_j^n-\bar{\mathbf U}_j^n
  =
  \begin{pmatrix}
    \Delta u_j^n\\[2pt]
    \Delta v_j^n
  \end{pmatrix},
  \qquad
  \Delta u_j^n:=u_j^n-\bar u_j^n,
  \qquad
  \Delta v_j^n:=v_j^n-\bar v_j^n.
\end{equation}
We also introduce the nonequilibrium variable
\begin{equation}
\label{eq:w-def-disc}
  w_j^n:=v_j^n-a u_j^n.
\end{equation}
Since \(\bar v_j^n=a\bar u_j^n\), one has
\begin{equation}
\label{eq:delta-v-w-du}
  \Delta v_j^n=w_j^n+a\Delta u_j^n.
\end{equation}
The discrete relative entropy is defined by
\begin{equation}
\label{eq:eta-rel-quadratic-disc}
  \eta_{j,\mathrm{rel}}^{n}
  :=
  \eta(\mathbf U_j^n)
  -
  \eta(\bar{\mathbf U}_j^n)
  -
  \nabla\eta(\bar{\mathbf U}_j^n)
  \cdot
  \big(
    \mathbf U_j^n-\bar{\mathbf U}_j^n
  \big).
\end{equation}
Since \(\eta\) is quadratic, this reduces to
\[
  \eta_{j,\mathrm{rel}}^{n}
  =
  \frac12
  (\Delta\mathbf U_j^n)^\top H\Delta\mathbf U_j^n
  =
  \frac12
  \|\Delta\mathbf U_j^n\|_H^2.
\]
We then introduce the global discrete relative entropy functional
\begin{equation}
\label{eq:phi-discret}
  \Phi_{\mathrm{rel}}^n
  :=
  \dx
  \sum_{j\in\mathbb Z}
  \eta_{j,\mathrm{rel}}^n
  =
  \frac{\dx}{2}
  \sum_{j\in\mathbb Z}
  \|\Delta\mathbf U_j^n\|_H^2.
\end{equation}
It is nonnegative under the subcharacteristic condition
\eqref{eq:subchar-final}.
The corresponding discrete relative entropy flux is defined by
\begin{equation}
\label{eq:Psi-rel-tzavaras}
  \Psi_{j+\frac12,\mathrm{rel}}^{n}
  :=
  \frac12
  (\Delta\mathbf U_j^n)^\top
  H A
  \Delta\mathbf U_{j+1}^n.
\end{equation}
%%%%%%%%%%%%%%%%%%%%%%%%%%%%%%%%%%%%%%%%%%%%%%%%%%%%%%%%%%%%%%%%%%%%%%%%%%%
We now establish the fully discrete counterpart of
Proposition~\ref{prop:relative-entropy-JX}. The proof combines an exact local
relative entropy identity, spatial cancellations, and estimates of the
remaining coupling terms. It follows the relative entropy strategy for
relaxation limits, adapted to the fully discrete setting
\cite{BerthonBessemoulinMathis16,BessemoulinMathis24,
BouchutBook04,BulteauBerthonBessemoulin19,DafermosBook10,
LattanzioTzavaras12,Tadmor2003}.

\begin{Thm}[Discrete relative entropy inequality]
\label{thm:global-disc-rel-entropy}
Let $(\mathbf U_j^n)_{j\in\mathbb Z,n\in\mathbb N}$ be the solution generated
by the Lie splitting scheme
\eqref{eq:conv-step}-\eqref{eq:source-step} for the Jin-Xin model, and let
$(\bar u_j^n)_{j\in\mathbb Z,n\in\mathbb N}$ be the solution of the limiting
scheme \eqref{eq:limit-scheme}. Set
\[
  \bar{\mathbf U}_j^n
  :=
  (\bar u_j^n,a\bar u_j^n)^\top,
\]
and let \(w_j^n\) be defined by \eqref{eq:w-def-disc}.

Assume that:
\begin{enumerate}[
  label=\textnormal{(H\arabic*)},
  ref=\textnormal{(H\arabic*)},
  start=0,
  leftmargin=*
]
  \item \label{ass:split-H0}
  The initial relative entropy is finite:
  \[
    \Phi_{\mathrm{rel}}^0<+\infty.
  \]

  \item \label{ass:split-H1}
  There exist two constant equilibrium states $\mathbf U_\pm$ such that, for
  every $n\in\mathbb N$,
  \begin{equation}
  \label{eq:states-at-infinity-disc}
    \lim_{j\to\pm\infty}\mathbf U_j^n=\mathbf U_\pm,
    \qquad
    \lim_{j\to\pm\infty}\bar{\mathbf U}_j^n=\mathbf U_\pm.
  \end{equation}

  \item \label{ass:split-H2}
  There exists a constant $K>0$ such that, for every $N$ satisfying
  $N\Delta t\leq T$,
  \begin{equation}
  \label{eq:H2-discrete-gradient-bound}
    \sum_{n=0}^{N-1}\Delta t\,\Delta x
    \sum_{j\in\mathbb Z}
    \left(
      (\tilde\delta_x\bar u)_{j+\frac12}^n
    \right)^2
    \leq K^2.
  \end{equation}

  \item \label{ass:split-H3}
  There exists a sufficiently small constant
  $\theta_0=\theta_0(a,\lambda)\in(0,1)$ such that the strengthened CFL
  condition
  \begin{equation}
  \label{eq:cfl-renforcee}
    \lambda\sigma\leq\theta_0
  \end{equation}
  holds.
\end{enumerate}

Then, for every $N$ satisfying $N\Delta t\leq T$, one has
\begin{equation}
\label{eq:global-disc-rel-entropy-at-N}
\begin{aligned}
  \Phi_{\mathrm{rel}}^{N}
  \leq\,
  \Phi_{\mathrm{rel}}^{0}
  &+
  \frac{\Delta t}{2\varepsilon}\,
  \Delta x\sum_{j\in\mathbb Z}(w_j^{0})^2
  +
  \frac{\Delta t^2}{2\varepsilon^2}\,
  \Delta x\sum_{j\in\mathbb Z}(w_j^{0})^2 +
  \frac{\varepsilon(\lambda^2-a^2)^2}{2}\,K^2.
\end{aligned}
\end{equation}
\end{Thm}

The proof of Theorem~\ref{thm:global-disc-rel-entropy} relies on auxiliary
lemmas establishing the local relative entropy identity and estimating the
remaining coupling terms.
%%%%%%%%%%%%%%%%%%%%%%%%%%%%%%%%%%%%%
%-
\begin{Lem}[Exact local discrete relative entropy identity]
\label{lem:disc-rel-entropy-local-exact}

Let $\bar{\mathbf U}_j^n=(\bar u_j^n,a\bar u_j^n)^\top\in\mathcal M_{\rm eq}$
for all $j\in\mathbb Z$ and $n\in\mathbb N$, and assume that it satisfies the
discrete entropy identity at equilibrium \eqref{eq:disc-eta-eq-identity}. Then,
for all $j\in\mathbb Z$ and $n\in\mathbb N$,
\begin{equation}\label{eq:disc-rel-entropy-local-identity-exact}
\begin{aligned}
  \eta_{j,\mathrm{rel}}^{n+1}
  &+
  \sigma\Big(
      \Psi_{j+\frac12,\mathrm{rel}}^{n}
      -
      \Psi_{j-\frac12,\mathrm{rel}}^{n}
    \Big)
  +
  \sigma\Big(
      \Gamma_{j+\frac12}^{n}
      -
      \Gamma_{j-\frac12}^{n}
    \Big)\\
  &=
  \eta_{j,\mathrm{rel}}^{n}
  -
  \sigma\big(R_j^{1,n}-\bar R_j^{1,n}\big)
  -
  \sigma\big(R_j^{2,n}-\bar R_j^{2,n}\big)
  +
  \frac{\sigma^2}{2}\big(R_j^{3,n}-\bar R_j^{3,n}\big)
  -
  R_j^{4,n}\\
  &\quad
  -
  \frac{\dt}{\varepsilon}
  \big(\Delta v_j^{n+1}-a\,\Delta u_j^{n+1}\big)^2
  -
  \frac{\dt^2}{2\varepsilon^2}
  \big(\Delta v_j^{n+1}-a\,\Delta u_j^{n+1}\big)^2,
\end{aligned}
\end{equation}
where
\begin{equation}\label{eq:R4-def-local-v2}
  R_j^{4,n}
  :=
  \nabla\eta(\bar{\mathbf U}_j^{n+1})\cdot\Delta\mathbf U_j^{n+1}
  -
  \nabla\eta(\bar{\mathbf U}_j^{n})\cdot\Delta\mathbf U_j^{n},
\end{equation}
and
\begin{equation}\label{eq:Gamma-def}
  \Gamma_{j+\frac12}^{n}
  :=
  \frac12\Big(
     \nabla\eta(\bar{\mathbf U}_{j}^{n})\cdot A\,\Delta\mathbf U_{j+1}^{n}
     +
     \nabla\eta(\bar{\mathbf U}_{j+1}^{n})\cdot A\,\Delta\mathbf U_{j}^{n}
  \Big).
\end{equation}
\end{Lem}

\begin{proof}
Subtracting \eqref{eq:disc-eta-eq-identity} from
\eqref{eq:full-disc-eta-identity}, and using that the reference state lies on
the equilibrium manifold, we have
\begin{equation}
\label{eq:local-relative-identities}
\begin{aligned}
  v_j^{n+1}-a u_j^{n+1}
  &=
  \Delta v_j^{n+1}-a\Delta u_j^{n+1},\\
  \eta(\mathbf U_j^m)-\eta(\bar{\mathbf U}_j^m)
  &=
  \eta_{j,\mathrm{rel}}^m
  +
  \nabla\eta(\bar{\mathbf U}_j^m)
  \cdot\Delta\mathbf U_j^m,
  \qquad m=n,n+1,\\
  \Psi_{j+\frac12}^n-\bar\Psi_{j+\frac12}^n
  &=
  \Psi_{j+\frac12,\mathrm{rel}}^n
  +
  \Gamma_{j+\frac12}^n.
\end{aligned}
\end{equation}
The second identity follows from
\eqref{eq:eta-rel-quadratic-disc}, while the last one follows from
\(\mathbf U_j^n=\bar{\mathbf U}_j^n+\Delta\mathbf U_j^n\) and the symmetry of
\(HA\). The second relation produces the increment \(R_j^{4,n}\).
Substituting \eqref{eq:local-relative-identities} into the difference of the
two entropy identities yields
\eqref{eq:disc-rel-entropy-local-identity-exact}.
\end{proof}
%--------------------------------------
%==========================================================
\begin{Lem}
\label{lem:R1minusbarR1-exact}
Assume the subcharacteristic condition \eqref{eq:subchar-final}, and let
\(\mathbf U_j^n\), \(\bar{\mathbf U}_j^n\), and
\(\Delta\mathbf U_j^n\) be as defined above. We define the convective
relative dissipation by
\begin{equation}
\label{eq:R1_rel}
  R_{j,\mathrm{rel}}^{1,n}
  :=
  \frac{\lambda}{4}
  \left(
    \|\Delta\mathbf U_{j+1}^n-\Delta\mathbf U_j^n\|_H^2
    +
    \|\Delta\mathbf U_j^n-\Delta\mathbf U_{j-1}^n\|_H^2
  \right).
\end{equation}
Since \(H\) is positive definite, one has
\(R_{j,\mathrm{rel}}^{1,n}\geq0\).

Then, for every \(n\in\mathbb N\),
\begin{equation}
\label{eq:R1minusbarR1-sum-exact}
\begin{aligned}
\Delta x\sum_{j\in\mathbb Z}
\left(
  R_j^{1,n}
  -
  \bar R_j^{1,n}
\right)
&=
\Delta x\sum_{j\in\mathbb Z}
R_{j,\mathrm{rel}}^{1,n}\\
&\quad
+
\lambda(\lambda^2-a^2)\,
\Delta x\sum_{j\in\mathbb Z}
(\bar u_{j+1}^n-\bar u_j^n)
(\Delta u_{j+1}^n-\Delta u_j^n),
\end{aligned}
\end{equation}
where \(R_j^{1,n}\) and \(\bar R_j^{1,n}\) are defined in
\eqref{eq:Di-def} and \eqref{eq:barDi-def}, respectively.
\end{Lem}

\begin{proof}
We use the polarization identity
\[
  \|\mathbf X+\mathbf Y\|_H^2
  =
  \|\mathbf X\|_H^2
  +
  \|\mathbf Y\|_H^2
  +
  2\mathbf Y^\top H\mathbf X,
  \qquad
  \mathbf X,\mathbf Y\in\mathbb R^2.
\]
Applying this identity to the two interface contributions in
\(R_j^{1,n}\), with
\[
  \mathbf U_{j+1}^n-\mathbf U_j^n
  =
  (\bar{\mathbf U}_{j+1}^n-\bar{\mathbf U}_j^n)
  +
  (\Delta\mathbf U_{j+1}^n-\Delta\mathbf U_j^n),
\]
and similarly at the interface \(j-\frac12\), gives
\[
\begin{aligned}
  R_j^{1,n}-\bar R_j^{1,n}
  &=
  R_{j,\mathrm{rel}}^{1,n}
  +
  \frac{\lambda}{2}
  (\Delta\mathbf U_{j+1}^n-\Delta\mathbf U_j^n)^\top
  H
  (\bar{\mathbf U}_{j+1}^n-\bar{\mathbf U}_j^n)\\
  &\quad
  +
  \frac{\lambda}{2}
  (\Delta\mathbf U_j^n-\Delta\mathbf U_{j-1}^n)^\top
  H
  (\bar{\mathbf U}_j^n-\bar{\mathbf U}_{j-1}^n).
\end{aligned}
\]
The equilibrium terms cancel with \(\bar R_j^{1,n}\), while the quadratic
terms involving \(\Delta\mathbf U^n\) give
\(R_{j,\mathrm{rel}}^{1,n}\).

Since
\[
  \bar{\mathbf U}_{j+1}^n-\bar{\mathbf U}_j^n
  =
  (\bar u_{j+1}^n-\bar u_j^n)
  \begin{pmatrix}
    1\\
    a
  \end{pmatrix},
  \qquad
  H
  \begin{pmatrix}
    1\\
    a
  \end{pmatrix}
  =
  \begin{pmatrix}
    \lambda^2-a^2\\
    0
  \end{pmatrix},
\]
the first cross term reduces to
\[
  \frac{\lambda(\lambda^2-a^2)}{2}
  (\bar u_{j+1}^n-\bar u_j^n)
  (\Delta u_{j+1}^n-\Delta u_j^n),
\]
and the second one has the analogous form at the interface
\(j-\frac12\).

After multiplying by \(\Delta x\) and summing over \(j\in\mathbb Z\), a change of summation index shows that the two cross-term sums are equal. Hence,
\eqref{eq:R1minusbarR1-sum-exact} follows.
\end{proof}
%==========================================================
%==========================================================
\begin{Lem}
\label{lem:convective-combination-clean}
Assume the subcharacteristic condition \eqref{eq:subchar-final} and
assumptions~\ref{ass:split-H1} and \ref{ass:split-H3} of
Theorem~\ref{thm:global-disc-rel-entropy}. If the CFL parameter also satisfies
\eqref{eq:theta0-assumption}, then there exists \(c_0\in(0,1)\), depending only
on \((a,\lambda)\) and \(\theta_0\), such that, for all \(n\in\mathbb N\),
\begin{equation}\label{eq:convective-combination-est}
\begin{aligned}
&-\sigma\,\dx\sum_{j\in\mathbb Z}\big(R_j^{1,n}-\bar R_j^{1,n}\big)
+\frac{\sigma^2}{2}\,\dx\sum_{j\in\mathbb Z}\big(R_j^{3,n}-\bar R_j^{3,n}\big)
-\dx\sum_{j\in\mathbb Z}R_j^{4,n}\\
&\le
-c_0\,\sigma\,\dx\sum_{j\in\mathbb Z}R_{j,\mathrm{rel}}^{1,n}
+\frac{\dt}{2\varepsilon}\,\dx\sum_{j\in\mathbb Z}(w_j^{n})^2
+\frac{(\lambda^2-a^2)^2\,\varepsilon}{2}\,
\dt\,\dx\sum_{j\in\mathbb Z}
\Big((\tilde\delta_x\bar u)_{j+\frac12}^n\Big)^2,
\end{aligned}
\end{equation}
where $R_j^{1,n}$ and $\bar R_j^{1,n}$ are defined in \eqref{eq:Di-def} and
\eqref{eq:barDi-def}, $R_j^{3,n}$ and $\bar R_j^{3,n}$ in \eqref{eq:Xi-def} and
\eqref{eq:barXi-def}, $R_j^{4,n}$ in \eqref{eq:R4-def-local-v2}, and
$R_{j,\mathrm{rel}}^{1,n}$ in \eqref{eq:R1_rel}.
\end{Lem}

\begin{proof}
We use the polarization identity
\begin{equation}\label{eq:polar-Q-again}
\|\mathbf X\|_H^2-\|\mathbf Y\|_H^2
=
\|\mathbf X-\mathbf Y\|_H^2
+
2(\mathbf X-\mathbf Y)^\top H\mathbf Y.
\end{equation}

We define the relative numerical flux by
\begin{equation}
\label{eq:Fhat-rel-def}
  \widehat{\mathbf F}_{j+\frac12,\mathrm{rel}}^n
  :=
  \widehat{\mathbf F}_{j+\frac12}^n
  -
  \widehat{\bar{\mathbf F}}_{j+\frac12}^n.
\end{equation}
Consequently,
\begin{equation}
\label{eq:DeltaF-rel-def}
\begin{aligned}
  \Delta\mathbf F_{j,\mathrm{rel}}^n
  &:=
  \Delta\mathbf F_j^n
  -
  \Delta\bar{\mathbf F}_j^n\\
  &=
  \widehat{\mathbf F}_{j+\frac12,\mathrm{rel}}^n
  -
  \widehat{\mathbf F}_{j-\frac12,\mathrm{rel}}^n.
\end{aligned}
\end{equation}
Since \(R_j^{3,n}=\|\Delta\mathbf F_j^n\|_H^2\) and
\(\bar R_j^{3,n}=\|\Delta\bar{\mathbf F}_j^n\|_H^2\), we obtain
\begin{equation}\label{eq:polar-R3-rel}
R_j^{3,n}-\bar R_j^{3,n}
=
\|\Delta\mathbf F_{j,\mathrm{rel}}^n\|_H^2
+
2(\Delta\mathbf F_{j,\mathrm{rel}}^n)^\top H\Delta\bar{\mathbf F}_j^n.
\end{equation}
Hence
\begin{equation}\label{eq:split-sigma2R3}
\begin{aligned}
\frac{\sigma^2}{2}\dx\sum_j(R_j^{3,n}-\bar R_j^{3,n})
=
\frac{\sigma^2}{2}
\dx\sum_j\|\Delta\mathbf F_{j,\mathrm{rel}}^n\|_H^2
+\sigma^2\dx\sum_j
(\Delta\mathbf F_{j,\mathrm{rel}}^n)^\top H\Delta\bar{\mathbf F}_j^n.
\end{aligned}
\end{equation}

For the Rusanov flux, one has the exact decomposition
\begin{equation}\label{eq:DFrel-exact}
\Delta\mathbf F_{j,\mathrm{rel}}^n
=
\frac12(A-\lambda I_2)(\Delta\mathbf U_{j+1}^n-\Delta\mathbf U_j^n)
+
\frac12(A+\lambda I_2)(\Delta\mathbf U_j^n-\Delta\mathbf U_{j-1}^n).
\end{equation}
Let
\[
C_0^{\pm}
:=\sup_{\mathbf Z\neq 0}\frac{\big\|\tfrac12(A\pm \lambda I_2)\mathbf Z\big\|_H^2}{\|\mathbf Z\|_H^2},
\qquad
C_0:=\max\{C_0^+,\,C_0^-\}.
\]
Then
\begin{equation}\label{eq:C0-bound}
\big\|\tfrac12(A\pm\lambda I_2)\mathbf Z\big\|_H^2
\le C_0\|\mathbf Z\|_H^2.
\end{equation}
Using \eqref{eq:DFrel-exact}, the inequality
\(\|X+Y\|_H^2\le2\|X\|_H^2+2\|Y\|_H^2\), and a shift of indices, we get
\begin{equation}\label{eq:DFrel-pointwise-bound}
\|\Delta\mathbf F_{j,\mathrm{rel}}^n\|_H^2
\le
2C_0
\Big(
\|\Delta\mathbf U_{j+1}^n-\Delta\mathbf U_j^n\|_H^2
+
\|\Delta\mathbf U_j^n-\Delta\mathbf U_{j-1}^n\|_H^2
\Big),
\end{equation}
and therefore
\begin{equation}\label{eq:DFrel-summed-bound}
\dx\sum_j\|\Delta\mathbf F_{j,\mathrm{rel}}^n\|_H^2
\le
4C_0\dx\sum_j
\|\Delta\mathbf U_{j+1}^n-\Delta\mathbf U_j^n\|_H^2.
\end{equation}
By the definition of \(R_{j,\mathrm{rel}}^{1,n}\),
\begin{equation}\label{eq:sumDFrel-to-R1rel}
\dx\sum_j\|\Delta\mathbf F_{j,\mathrm{rel}}^n\|_H^2
\le
\frac{8C_0}{\lambda}
\dx\sum_j R_{j,\mathrm{rel}}^{1,n}.
\end{equation}
Thus
\begin{equation}\label{eq:split-sigma2R3final}
\begin{aligned}
\frac{\sigma^2}{2}\dx\sum_j(R_j^{3,n}-\bar R_j^{3,n})
\le
\frac{\sigma^2}{2}\frac{8C_0}{\lambda}
\dx\sum_j R_{j,\mathrm{rel}}^{1,n}
+\sigma^2\dx\sum_j
(\Delta\mathbf F_{j,\mathrm{rel}}^n)^\top H\Delta\bar{\mathbf F}_j^n.
\end{aligned}
\end{equation}

We now treat the cross term together with \(R_j^{4,n}\). From the convective
updates,
\begin{equation}\label{eq:Uhalf_minus_Ubar_explicit}
\mathbf U_j^{n+\frac12}-\bar{\mathbf U}_j^{n+1}
=
\Delta\mathbf U_j^n
-
\sigma\Delta\mathbf F_{j,\mathrm{rel}}^n.
\end{equation}
Equivalently,
\begin{equation}\label{eq:cross_term_identity_explicit}
\sigma^2
(\Delta\mathbf F_{j,\mathrm{rel}}^n)^\top H\Delta\bar{\mathbf F}_j^n
=
\sigma
\Big(
\Delta\mathbf U_j^n
-
(\mathbf U_j^{n+\frac12}-\bar{\mathbf U}_j^{n+1})
\Big)^\top H\Delta\bar{\mathbf F}_j^n.
\end{equation}
Moreover, the source step does not change the first component. Since
\(\nabla\eta(\bar{\mathbf U}_j^m)=((\lambda^2-a^2)\bar u_j^m,0)^\top\), we can
write
\begin{equation}\label{eq:R4_convective_form_explicit}
R_j^{4,n}
=
\nabla\eta(\bar{\mathbf U}_j^{n+1})
\cdot(\mathbf U_j^{n+\frac12}-\bar{\mathbf U}_j^{n+1})
-
\nabla\eta(\bar{\mathbf U}_j^{n})
\cdot\Delta\mathbf U_j^n.
\end{equation}
Using
\(\bar{\mathbf U}_j^{n+1}
=\bar{\mathbf U}_j^n-\sigma\Delta\bar{\mathbf F}_j^n\), we have
\[
\nabla\eta(\bar{\mathbf U}_j^{n+1})
=
\nabla\eta(\bar{\mathbf U}_j^n)
-
\sigma H\Delta\bar{\mathbf F}_j^n.
\]
Substitution in \eqref{eq:R4_convective_form_explicit} gives
\begin{equation}\label{eq:R4_expand_explicit}
\begin{aligned}
R_j^{4,n}
&=
\nabla\eta(\bar{\mathbf U}_j^n)
\cdot
\Big(
\mathbf U_j^{n+\frac12}-\bar{\mathbf U}_j^{n+1}
-
\Delta\mathbf U_j^n
\Big)\\
&\quad
-\sigma
(\mathbf U_j^{n+\frac12}-\bar{\mathbf U}_j^{n+1})^\top
H\Delta\bar{\mathbf F}_j^n.
\end{aligned}
\end{equation}
Combining \eqref{eq:cross_term_identity_explicit} and
\eqref{eq:R4_expand_explicit}, and using
\eqref{eq:Uhalf_minus_Ubar_explicit}, yields
\begin{equation}\label{eq:key_identity_pointwise_explicit}
\begin{aligned}
\sigma^2
(\Delta\mathbf F_{j,\mathrm{rel}}^n)^\top H\Delta\bar{\mathbf F}_j^n
-
R_j^{4,n}
=
\sigma(\Delta\mathbf U_j^n)^\top H\Delta\bar{\mathbf F}_j^n
+
\sigma\nabla\eta(\bar{\mathbf U}_j^n)
\cdot\Delta\mathbf F_{j,\mathrm{rel}}^n.
\end{aligned}
\end{equation}

Summing \eqref{eq:key_identity_pointwise_explicit} over \(j\) and applying
discrete summation by parts, we obtain
\begin{equation}
\label{eq:discret-sum}
\begin{aligned}
  \sum_{j\in\mathbb Z}
  \nabla\eta(\bar{\mathbf U}_j^n)
  \cdot
  \left(
    \widehat{\mathbf F}_{j+\frac12,\mathrm{rel}}^n
    -
    \widehat{\mathbf F}_{j-\frac12,\mathrm{rel}}^n
  \right)=
  -
  \sum_{j\in\mathbb Z}
  \left(
    \nabla\eta(\bar{\mathbf U}_{j+1}^n)
    -
    \nabla\eta(\bar{\mathbf U}_j^n)
  \right)
  \cdot
  \widehat{\mathbf F}_{j+\frac12,\mathrm{rel}}^n,
\end{aligned}
\end{equation}
where the boundary terms vanish by \ref{ass:split-H1}. Since
\[
\nabla\eta(\bar{\mathbf U}_{j+1}^n)
-
\nabla\eta(\bar{\mathbf U}_{j}^n)
=
\big((\lambda^2-a^2)(\bar u_{j+1}^n-\bar u_j^n),0\big)^\top,
\]
we obtain
\begin{equation}\label{eq:sbp-reduced-to-first-comp}
\begin{aligned}
&\big(
\nabla\eta(\bar{\mathbf U}_{j+1}^n)
-
\nabla\eta(\bar{\mathbf U}_{j}^n)
\big)
\cdot
\widehat{\mathbf F}_{j+\frac12,\mathrm{rel}}^n
=
(\lambda^2-a^2)(\bar u_{j+1}^n-\bar u_j^n)
\widehat{\mathbf F}_{j+\frac12,\mathrm{rel}}^{(1),n}.
\end{aligned}
\end{equation}
By linearity of the Rusanov flux and using
\(\Delta v_j^n=w_j^n+a\Delta u_j^n\),
\begin{equation}\label{eq:Fhat1rel-explicit-final}
\widehat{\mathbf F}_{j+\frac12,\mathrm{rel}}^{(1),n}
=
\frac12(w_j^n+w_{j+1}^n)
+
\frac{a+\lambda}{2}\Delta u_j^n
+
\frac{a-\lambda}{2}\Delta u_{j+1}^n.
\end{equation}

It remains to compute the first term in
\eqref{eq:key_identity_pointwise_explicit}. Since
\[
\Delta\bar{\mathbf F}_j^n=\Delta\widehat f_j^n(1,a)^\top,
\qquad
H(1,a)^\top=(\lambda^2-a^2,0)^\top,
\]
we have
\begin{equation}\label{eq:second-term}
(\Delta\mathbf U_j^n)^\top H\Delta\bar{\mathbf F}_j^n
=
(\lambda^2-a^2)\Delta u_j^n\Delta\widehat f_j^n.
\end{equation}
Moreover,
\begin{equation}\label{eq:flux-wide}
\Delta\widehat f_j^n
=
\frac{a-\lambda}{2}(\bar u_{j+1}^n-\bar u_j^n)
+
\frac{a+\lambda}{2}(\bar u_j^n-\bar u_{j-1}^n).
\end{equation}
Using a shift of indices, this gives
\begin{equation}\label{eq:DU-DbF-final}
\begin{aligned}
\dx\sum_j
(\Delta\mathbf U_j^n)^\top H\Delta\bar{\mathbf F}_j^n
&=
(\lambda^2-a^2)\dx\sum_j
(\bar u_{j+1}^n-\bar u_j^n)\\
&\quad\times
\left(
\frac{a-\lambda}{2}\Delta u_j^n
+
\frac{a+\lambda}{2}\Delta u_{j+1}^n
\right).
\end{aligned}
\end{equation}
Combining \eqref{eq:discret-sum}-\eqref{eq:DU-DbF-final}, we obtain
\begin{equation}\label{eq:sum-key-identity1}
\begin{aligned}
&\sigma^2\dx\sum_j
(\Delta\mathbf F_{j,\mathrm{rel}}^n)^\top H\Delta\bar{\mathbf F}_j^n
-
\dx\sum_j R_j^{4,n}\\
&=
-\sigma(\lambda^2-a^2)\dx\sum_j
(\bar u_{j+1}^n-\bar u_j^n)
\frac{w_j^n+w_{j+1}^n}{2}\\
&\quad
+
\sigma\lambda(\lambda^2-a^2)\dx\sum_j
(\bar u_{j+1}^n-\bar u_j^n)
(\Delta u_{j+1}^n-\Delta u_j^n).
\end{aligned}
\end{equation}
Adding
\(-\sigma\dx\sum_j(R_j^{1,n}-\bar R_j^{1,n})\) and using
Lemma~\ref{lem:R1minusbarR1-exact}, the last term cancels exactly. Hence
\begin{equation}\label{eq:after-cancel}
\begin{aligned}
&-\sigma\dx\sum_j(R_j^{1,n}-\bar R_j^{1,n})
+
\sigma^2\dx\sum_j
(\Delta\mathbf F_{j,\mathrm{rel}}^n)^\top H\Delta\bar{\mathbf F}_j^n
-
\dx\sum_jR_j^{4,n}\\
&=
-\sigma\dx\sum_jR_{j,\mathrm{rel}}^{1,n}
-
\sigma(\lambda^2-a^2)\dx\sum_j
(\bar u_{j+1}^n-\bar u_j^n)
\frac{w_j^n+w_{j+1}^n}{2}.
\end{aligned}
\end{equation}

We estimate the remaining coupling term by Young's inequality. Since
\[
\bar u_{j+1}^n-\bar u_j^n
=
\dx(\tilde\delta_x\bar u)_{j+\frac12}^n,
\qquad
\sigma\dx=\dt,
\]
we get
\begin{equation}\label{eq:w-estimate-final}
\begin{aligned}
&\left|
\sigma(\lambda^2-a^2)\dx\sum_j
(\bar u_{j+1}^n-\bar u_j^n)
\frac{w_j^n+w_{j+1}^n}{2}
\right|\\
&\le
\frac{\dt}{2\varepsilon}\dx\sum_j(w_j^n)^2
+
\frac{(\lambda^2-a^2)^2\varepsilon}{2}
\dt\dx\sum_j
\Big((\tilde\delta_x\bar u)_{j+\frac12}^n\Big)^2.
\end{aligned}
\end{equation}

Finally, combining \eqref{eq:split-sigma2R3final} with
\eqref{eq:after-cancel} gives
\begin{equation}\label{eq:pre_absorb}
\begin{aligned}
&-\sigma\dx\sum_j(R_j^{1,n}-\bar R_j^{1,n})
+
\frac{\sigma^2}{2}\dx\sum_j(R_j^{3,n}-\bar R_j^{3,n})
-
\dx\sum_jR_j^{4,n}\\
&\le
-\sigma
\left(
1-\frac{4C_0}{\lambda}\sigma
\right)
\dx\sum_jR_{j,\mathrm{rel}}^{1,n}
-
\sigma(\lambda^2-a^2)\dx\sum_j
(\bar u_{j+1}^n-\bar u_j^n)
\frac{w_j^n+w_{j+1}^n}{2}.
\end{aligned}
\end{equation}
We choose \(\theta_0\) so that
\begin{equation}\label{eq:theta0-assumption}
  0<\theta_0\le\frac{\lambda^2}{8C_0}.
\end{equation}
Since \(\lambda\sigma\le\theta_0\), this gives
\[
  1-\frac{4C_0}{\lambda}\sigma\ge\frac12.
\]
Thus the estimate holds uniformly with the choice
\[
  c_0:=\frac12.
\]
Combining \eqref{eq:pre_absorb} with \eqref{eq:w-estimate-final} proves
\eqref{eq:convective-combination-est}.
\end{proof}
%==========================================================
% Proof of the global theorem
%==========================================================
\begin{proof}[Proof of Theorem~\ref{thm:global-disc-rel-entropy}]

Since \(\bar v_j^m=a\bar u_j^m\), relation
\eqref{eq:delta-v-w-du} gives
\[
  \Delta v_j^m-a\Delta u_j^m=w_j^m.
\]
Multiplying the local identity
\eqref{eq:disc-rel-entropy-local-identity-exact} by \(\dx\) and summing over
\(j\in\mathbb Z\), the discrete flux differences vanish under
assumption~\ref{ass:split-H1}. The same assumption also gives
\[
  \dx\sum_{j\in\mathbb Z}
  \left(
    R_j^{2,n}-\bar R_j^{2,n}
  \right)
  =0.
\]
Therefore,
\begin{equation}
\label{eq:global-step-n-exact-noR2}
\begin{aligned}
  \Phi_{\mathrm{rel}}^{n+1}
  &=
  \Phi_{\mathrm{rel}}^n
  -
  \sig\,\dx
  \sum_{j\in\mathbb Z}
  \left(
    R_j^{1,n}-\bar R_j^{1,n}
  \right)
  +
  \frac{\sig^2}{2}\,\dx
  \sum_{j\in\mathbb Z}
  \left(
    R_j^{3,n}-\bar R_j^{3,n}
  \right)
  -
  \dx\sum_{j\in\mathbb Z}R_j^{4,n}\\
  &\quad
  -
  \frac{\dt}{\varepsilon}\,
  \dx\sum_{j\in\mathbb Z}(w_j^{n+1})^2
  -
  \frac{\dt^2}{2\varepsilon^2}\,
  \dx\sum_{j\in\mathbb Z}(w_j^{n+1})^2.
\end{aligned}
\end{equation}
Applying Lemma~\ref{lem:convective-combination-clean} yields
\begin{equation}
\label{eq:global-step-n-ineq-detailed}
\begin{aligned}
  \Phi_{\mathrm{rel}}^{n+1}
  &+
  \frac{\dt}{\varepsilon}\,
  \dx\sum_{j\in\mathbb Z}(w_j^{n+1})^2
  +
  \frac{\dt^2}{2\varepsilon^2}\,
  \dx\sum_{j\in\mathbb Z}(w_j^{n+1})^2+
  c_0\,\sig\,\dx
  \sum_{j\in\mathbb Z}R_{j,\mathrm{rel}}^{1,n}\\
  &\leq
  \Phi_{\mathrm{rel}}^n
  +
  \frac{\dt}{2\varepsilon}\,
  \dx\sum_{j\in\mathbb Z}(w_j^n)^2
  +
  \frac{(\lambda^2-a^2)^2\varepsilon}{2}\,
  \dt\,\dx
  \sum_{j\in\mathbb Z}
  \left(
    (\tilde\delta_x\bar u)_{j+\frac12}^n
  \right)^2.
\end{aligned}
\end{equation}
We introduce
\begin{equation}
\label{eq:def-short-Sn-ourcase}
  \mathfrak S^n
  :=
  \Phi_{\mathrm{rel}}^n
  +
  \frac{\dt}{2\varepsilon}\,
  \dx\sum_{j\in\mathbb Z}(w_j^n)^2
  +
  \frac{\dt^2}{2\varepsilon^2}\,
  \dx\sum_{j\in\mathbb Z}(w_j^n)^2.
\end{equation}
Since the last term in \(\mathfrak S^n\) is nonnegative,
\eqref{eq:global-step-n-ineq-detailed} implies
\begin{equation}
\label{eq:S-iter-detailed}
\begin{aligned}
  \mathfrak S^{n+1}
  &+
  \frac{\dt}{2\varepsilon}\,
  \dx\sum_{j\in\mathbb Z}(w_j^{n+1})^2
  +
  c_0\,\sig\,\dx
  \sum_{j\in\mathbb Z}R_{j,\mathrm{rel}}^{1,n}\\
  &\leq
  \mathfrak S^n
  +
  \frac{(\lambda^2-a^2)^2\varepsilon}{2}\,
  \dt\,\dx
  \sum_{j\in\mathbb Z}
  \left(
    (\tilde\delta_x\bar u)_{j+\frac12}^n
  \right)^2.
\end{aligned}
\end{equation}
Summing from \(n=0\) to \(N-1\) and using
assumption~\ref{ass:split-H2}, we obtain
\[
\begin{aligned}
  \mathfrak S^N
  +
  \frac{\dt}{2\varepsilon}
  \sum_{n=0}^{N-1}
  \dx\sum_{j\in\mathbb Z}(w_j^{n+1})^2 +
  c_0\,\sig
  \sum_{n=0}^{N-1}
  \dx\sum_{j\in\mathbb Z}R_{j,\mathrm{rel}}^{1,n} \leq
  \mathfrak S^0
  +
  \frac{(\lambda^2-a^2)^2\varepsilon}{2}\,K^2.
\end{aligned}
\]
Discarding the nonnegative terms on the left and using the definition of
\(\mathfrak S^0\) gives
\[
  \Phi_{\mathrm{rel}}^N
  \leq
  \Phi_{\mathrm{rel}}^0
  +
  \frac{\dt}{2\varepsilon}\,
  \dx\sum_{j\in\mathbb Z}(w_j^0)^2
  +
  \frac{\dt^2}{2\varepsilon^2}\,
  \dx\sum_{j\in\mathbb Z}(w_j^0)^2
  +
  \frac{(\lambda^2-a^2)^2\varepsilon}{2}\,K^2.
\]
This proves \eqref{eq:global-disc-rel-entropy-at-N}.
\end{proof}
The preceding result is the fully discrete counterpart of the continuous
relative entropy estimate. In particular, if the initial data are
well prepared,
then \(
  \Phi_{\mathrm{rel}}^N=O(\varepsilon)
\)
uniformly for \(N\dt\leq T\).
%%%%%%%%%%%%%%%%%%%%%%%%%%%%%%%%%%%%%%%%%%%%%%%%%%%%%%%%%
%%%%%%%%%%%%%%%%%%%%%%%%%%%%%%%%%%%%%%%%%%%%%%%%%%%%%%%%%%%%%%%%%%%%%%%%%%%%%%
% SECTION : Discrete relative entropy for the staggered scheme
%%%%%%%%%%%%%%%%%%%%%%%%%%%%%%%%%%%%%%%%%%%%%%%%%%%%%%%%%%%%%%%%%%%%%%%%%%%%%%
\section{Discrete relative entropy for the staggered scheme}
\label{sec:disc-rel-entropy-stag}

\subsection{Definition of the staggered scheme}
\label{subsec:stag-scheme}

We now consider the staggered \\ asymptotic-preserving scheme for the linear
Jin-Xin relaxation system. The scheme is introduced in \cite{MahmoudMathis2025}. The purpose of this
section is to analyze its discrete relative entropy structure by comparing the
relaxation solution with the corresponding limiting discrete solution.

We use the spatial and temporal discretizations introduced in
Section~\ref{sec:numerical-schemes}, with
\(
  \sigma=\frac{\Delta t}{\Delta x}.
\)
The numerical solution is defined at the cell centers \(x_j\) at integer time
levels \(t^n\), whereas the intermediate states are defined at the interfaces
\(x_{j+\frac12}\) at time \(t^{n+\frac12}\). We write
\[
  \mathbf U_j^n
  =
  \begin{pmatrix}
    u_j^n\\
    v_j^n
  \end{pmatrix},
  \qquad
  \mathbf U_{j+\frac12}^{n+\frac12}
  =
  \begin{pmatrix}
    u_{j+\frac12}^{n+\frac12}\\
    v_{j+\frac12}^{n+\frac12}
  \end{pmatrix}.
\]

The scheme consists of two successive staggered updates. The first update
advances the solution from the cell centers at time \(t^n\) to the interfaces
at time \(t^{n+\frac12}\), while the second update advances the intermediate
states back to the cell centers at time \(t^{n+1}\). For the Jin-Xin system
\eqref{eq:jin-xin-system}, the scheme is given by
\begin{equation}
\label{eq:disc-JX-stag}
\begin{aligned}
  u_{j-\frac12}^{n+\frac12}
  &=
  \frac{u_{j-1}^n+u_j^n}{2}
  -
  \frac{\sigma}{2}
  \Big(
    a u_j^n
    +
    e^{-\frac{\Delta t}{2\varepsilon}}
    (v_j^n-a u_j^n)
    -
    a u_{j-1}^n
    -
    e^{-\frac{\Delta t}{2\varepsilon}}
    (v_{j-1}^n-a u_{j-1}^n)
  \Big),
  \\[0.4em]
  v_{j-\frac12}^{n+\frac12}
  &=
  \left(
    1+\frac{\Delta t}{2\varepsilon}
  \right)^{-1}
  \left[
    \frac{v_{j-1}^n+v_j^n}{2}
    -
    \frac{\lambda^2\sigma}{2}
    (u_j^n-u_{j-1}^n)
    +
    \frac{\Delta t}{2\varepsilon}
    a\,u_{j-\frac12}^{n+\frac12}
  \right],
  \\[0.4em]
  u_j^{n+1}
  &=
  \frac{
    u_{j-\frac12}^{n+\frac12}
    +
    u_{j+\frac12}^{n+\frac12}
  }{2}
  -
  \frac{\sigma}{2}
  \Big(
    a u_{j+\frac12}^{n+\frac12}
    +
    e^{-\frac{\Delta t}{2\varepsilon}}
    \big(
      v_{j+\frac12}^{n+\frac12}
      -
      a u_{j+\frac12}^{n+\frac12}
    \big)
  \\
  &\hspace{8em}
    -
    a u_{j-\frac12}^{n+\frac12}
    -
    e^{-\frac{\Delta t}{2\varepsilon}}
    \big(
      v_{j-\frac12}^{n+\frac12}
      -
      a u_{j-\frac12}^{n+\frac12}
    \big)
  \Big),
  \\[0.4em]
  v_j^{n+1}
  &=
  \left(
    1+\frac{\Delta t}{2\varepsilon}
  \right)^{-1}
  \left[
    \frac{
      v_{j-\frac12}^{n+\frac12}
      +
      v_{j+\frac12}^{n+\frac12}
    }{2}
    -
    \frac{\lambda^2\sigma}{2}
    \big(
      u_{j+\frac12}^{n+\frac12}
      -
      u_{j-\frac12}^{n+\frac12}
    \big)
    +
    \frac{\Delta t}{2\varepsilon}
    a\,u_j^{n+1}
  \right].
\end{aligned}
\end{equation}

We use the nonequilibrium variable \(w_j^n\) introduced in
\eqref{eq:w-def-disc} and extend its definition to the intermediate staggered
states by setting
\begin{equation}
\label{eq:w-def-stag}
  w_{j+\frac12}^{n+\frac12}
  :=
  v_{j+\frac12}^{n+\frac12}
  -
  a\,u_{j+\frac12}^{n+\frac12}.
\end{equation}

In terms of the variables \((u,w)\), the two staggered updates for \(u\) can be
written as
\begin{equation}
\label{eq:u-half-compact}
\begin{aligned}
  u_{j-\frac12}^{n+\frac12}
  &=
  \frac{u_{j-1}^n+u_j^n}{2}
  -
  \frac{\sigma}{2}
  \left[
    a(u_j^n-u_{j-1}^n)
    +
    e^{-\frac{\Delta t}{2\varepsilon}}
    (w_j^n-w_{j-1}^n)
  \right],
  \\
  u_j^{n+1}
  &=
  \frac{
    u_{j-\frac12}^{n+\frac12}
    +
    u_{j+\frac12}^{n+\frac12}
  }{2}
  -
  \frac{\sigma}{2}
  \left[
    a
    \big(
      u_{j+\frac12}^{n+\frac12}
      -
      u_{j-\frac12}^{n+\frac12}
    \big)
    +
    e^{-\frac{\Delta t}{2\varepsilon}}
    \big(
      w_{j+\frac12}^{n+\frac12}
      -
      w_{j-\frac12}^{n+\frac12}
    \big)
  \right].
\end{aligned}
\end{equation}

The corresponding updates for the nonequilibrium variable are
\begin{equation}
\label{eq:w-half-u-w-form}
\begin{aligned}
  \left(
    1+\frac{\Delta t}{2\varepsilon}
  \right)
  w_{j-\frac12}^{n+\frac12}
  &=
  \frac{w_{j-1}^n+w_j^n}{2}
  -
  \frac{(\lambda^2-a^2)\sigma}{2}
  (u_j^n-u_{j-1}^n)
  +
  \frac{a\sigma}{2}
  e^{-\frac{\Delta t}{2\varepsilon}}
  (w_j^n-w_{j-1}^n),
  \\
  \left(
    1+\frac{\Delta t}{2\varepsilon}
  \right)
  w_j^{n+1}
  &=
  \frac{
    w_{j-\frac12}^{n+\frac12}
    +
    w_{j+\frac12}^{n+\frac12}
  }{2}
  -
  \frac{(\lambda^2-a^2)\sigma}{2}
  \big(
    u_{j+\frac12}^{n+\frac12}
    -
    u_{j-\frac12}^{n+\frac12}
  \big)
  \\
  &\quad
  +
  \frac{a\sigma}{2}
  e^{-\frac{\Delta t}{2\varepsilon}}
  \big(
    w_{j+\frac12}^{n+\frac12}
    -
    w_{j-\frac12}^{n+\frac12}
  \big).
\end{aligned}
\end{equation}

\medskip
\noindent\textbf{Limiting scheme and lifting.}
We now identify the corresponding limiting discrete scheme. For fixed
\(\Delta t>0\), as \(\varepsilon\to0\),
\[
  e^{-\frac{\Delta t}{2\varepsilon}}
  \longrightarrow0,
  \qquad
  \left(
    1+\frac{\Delta t}{2\varepsilon}
  \right)^{-1}
  \longrightarrow0.
\]
Consequently, the nonequilibrium variable tends to zero, and the numerical
solution is driven towards the equilibrium manifold:
\[
  w=0,
  \qquad\text{or equivalently}\qquad
  v=au.
\]
Passing formally to the limit in the two updates for \(u\) yields
\begin{equation}
\label{eq:ubar-half-stag}
\begin{aligned}
  \bar u_{j-\frac12}^{n+\frac12}
  &=
  \frac{\bar u_{j-1}^n+\bar u_j^n}{2}
  -
  \frac{a\sigma}{2}
  (\bar u_j^n-\bar u_{j-1}^n),
  \\
  \bar u_j^{n+1}
  &=
  \frac{
    \bar u_{j-\frac12}^{n+\frac12}
    +
    \bar u_{j+\frac12}^{n+\frac12}
  }{2}
  -
  \frac{a\sigma}{2}
  \big(
    \bar u_{j+\frac12}^{n+\frac12}
    -
    \bar u_{j-\frac12}^{n+\frac12}
  \big).
\end{aligned}
\end{equation}
The scalar limiting solution is lifted onto the equilibrium manifold at both
the integer and intermediate time levels by setting
\begin{equation}
\label{eq:Ubar-def-stag}
  \bar{\mathbf U}_j^n
  :=
  \begin{pmatrix}
    \bar u_j^n\\
    a\bar u_j^n
  \end{pmatrix},
  \qquad
  \bar{\mathbf U}_{j+\frac12}^{n+\frac12}
  :=
  \begin{pmatrix}
    \bar u_{j+\frac12}^{n+\frac12}\\
    a\bar u_{j+\frac12}^{n+\frac12}
  \end{pmatrix}.
\end{equation}
These lifted states belong to \(\mathcal M_{\mathrm{eq}}\) and provide the
discrete reference solution for the staggered relative entropy analysis.
%%%%%%%%%%%%%%%%%%%%%%%%%%%%%%%%%%%%%%%%%%%%
\subsection{Discrete relative entropy}
\label{subsec:stag-rel-entropy}

We use the \(H\)-norm introduced in the previous section. Since the reference
state belongs to the equilibrium manifold, for
\(\bar{\mathbf U}=(\bar u,a\bar u)^\top\), one has
\begin{equation}
\label{eq:Q-split-w-stag}
\begin{aligned}
  \|\mathbf U-\bar{\mathbf U}\|_H^2 =
  (\lambda^2-a^2)(u-\bar u)^2+w^2.
\end{aligned}
\end{equation}
Thus, the relative entropy controls both the error in \(u\) and the
nonequilibrium variable \(w\).

We define the local relative entropies at the integer and intermediate time
levels by
\begin{equation}\label{eq:entropy-stagg}
  \eta_{j,\mathrm{rel}}^n
  :=
  \frac12
  \|\mathbf U_j^n-\bar{\mathbf U}_j^n\|_H^2,
  \qquad
  \eta_{j+\frac12,\mathrm{rel}}^{n+\frac12}
  :=
  \frac12
  \left\|
    \mathbf U_{j+\frac12}^{n+\frac12}
    -
    \bar{\mathbf U}_{j+\frac12}^{n+\frac12}
  \right\|_H^2.
\end{equation}
The corresponding global relative entropy functionals are
\begin{equation}
\label{eq:phi-stag}
  \Phi_{\mathrm{rel}}^n
  :=
  \Delta x
  \sum_{j\in\mathbb Z}
  \eta_{j,\mathrm{rel}}^n,
  \qquad
  \Phi_{\mathrm{rel}}^{n+\frac12}
  :=
  \Delta x
  \sum_{j\in\mathbb Z}
  \eta_{j+\frac12,\mathrm{rel}}^{n+\frac12}.
\end{equation}

We also introduce the discrete derivatives of the reference solution at the
two staggered levels:
\begin{equation}
\label{eq:tilde-dx-ubar-stag}
  (\tilde\delta_x\bar u)_{j+\frac12}^n
  :=
  \frac{\bar u_{j+1}^n-\bar u_j^n}{\Delta x},
  \qquad
  (\tilde\delta_x\bar u)_j^{n+\frac12}
  :=
  \frac{
    \bar u_{j+\frac12}^{n+\frac12}
    -
    \bar u_{j-\frac12}^{n+\frac12}
  }{\Delta x}.
\end{equation}

\begin{Thm}[Global discrete relative entropy estimate for the staggered scheme]
\label{thm:global-stag-final}
Let \((\mathbf U_j^n)_{j,n}\) be generated by the staggered scheme
\eqref{eq:disc-JX-stag}, and let \((\bar u_j^n)_{j,n}\) be the solution of the
limiting staggered scheme \eqref{eq:ubar-half-stag}. Let
\(\bar{\mathbf U}\) be defined by \eqref{eq:Ubar-def-stag}, and let \(w\)
denote the nonequilibrium variable introduced above.

Assume the subcharacteristic condition \eqref{eq:subchar-final} and the CFL
condition
\begin{equation}
\label{eq:cfl-global-final}
  \lambda\sigma\leq\frac{1}{\sqrt{3}}.
\end{equation}
Define
\begin{equation}
\label{eq:kappa-sigma-def-final}
  \kappa_\sigma
  :=
  \frac{
    (\lambda^2-a^2)\sigma^2
  }{
    2(1-\lambda^2\sigma^2)
  }.
\end{equation}
Assume moreover that:
\begin{enumerate}[
  label=\textnormal{(H\arabic*)},
  ref=\textnormal{(H\arabic*)},
  start=0,
  leftmargin=*
]
  \item
  \label{ass:stag-H0}
  \(
    \Phi_{\mathrm{rel}}^0<+\infty;
  \)

  \item
  \label{ass:stag-H1}
  there exist two constant equilibrium states
  \[
    \mathbf U_\pm
    =
    \begin{pmatrix}
      u_\pm\\
      a u_\pm
    \end{pmatrix}
  \]
  such that, for every \(n\in\mathbb N\),
  \begin{equation}
  \label{eq:states-at-infinity-stag}
    \lim_{j\to\pm\infty}\mathbf U_j^n
    =
    \mathbf U_\pm,
    \qquad
    \lim_{j\to\pm\infty}\bar{\mathbf U}_j^n
    =
    \mathbf U_\pm;
  \end{equation}

  \item
  \label{ass:stag-H2}
  there exists \(K>0\) such that, for every \(N\) satisfying
  \(N\Delta t\leq T\),
  \begin{equation}
  \label{eq:stag-H2}
    \sum_{n=0}^{N-1}
    \Delta t\,\Delta x
    \sum_{j\in\mathbb Z}
    \left(
      (\tilde\delta_x\bar u)_{j+\frac12}^n
    \right)^2
    \leq K^2.
  \end{equation}
\end{enumerate}
Then, for every \(N\) satisfying \(N\Delta t\leq T\),
\begin{equation}
\label{eq:global-bound-final}
\begin{aligned}
  \Phi_{\mathrm{rel}}^N
  +
  \kappa_\sigma
  \frac{\Delta t}{\varepsilon}
  \Delta x
  \sum_{j\in\mathbb Z}
  (w_j^N)^2
  \leq
  \Phi_{\mathrm{rel}}^0
  +
  \kappa_\sigma
  \frac{\Delta t}{\varepsilon}
  \Delta x
  \sum_{j\in\mathbb Z}
  (w_j^0)^2
  +
  \frac{(\lambda^2-a^2)^2}{2}
  \varepsilon K^2.
\end{aligned}
\end{equation}
\end{Thm}

The proof is based on two local relative entropy estimates corresponding to
the two staggered updates.

\begin{Lem}[Local half-step relative entropy inequality]
\label{lem:local-half-stag-corrected}
Assume the subcharacteristic condition \eqref{eq:subchar-final} and the CFL
condition \eqref{eq:cfl-global-final}. Then, for every \(j\in\mathbb Z\) and
\(n\in\mathbb N\),
\begin{align}
\label{eq:local-half-ineq-stag-corrected}
  \eta_{j-\frac12,\mathrm{rel}}^{n+\frac12}
  &\leq
  \frac12
  \left(
    \eta_{j-1,\mathrm{rel}}^n
    +
    \eta_{j,\mathrm{rel}}^n
  \right)
  -
  \frac{\sigma}{2}
  \left(
    \mathcal B_{j,\mathrm{rel}}^n
    -
    \mathcal B_{j-1,\mathrm{rel}}^n
  \right)
  -
  \frac{\Delta t}{4\varepsilon}
  \left(
    w_{j-\frac12}^{n+\frac12}
  \right)^2
  \notag\\
  &\quad
  +
  \kappa_\sigma
  \frac{\Delta t}{\varepsilon}
  \frac{
    (w_{j-1}^n)^2+(w_j^n)^2
  }{2}
  +
  \frac{(\lambda^2-a^2)^2\varepsilon}{4}
  \Delta t
  \left(
    (\tilde\delta_x\bar u)_{j-\frac12}^n
  \right)^2,
\end{align}
where
\[
\begin{aligned}
  \mathcal B_{k,\mathrm{rel}}^n
  :=
  \frac{a(\lambda^2-a^2)}{2}
  (\Delta u_k^n)^2
  +
  (\lambda^2-a^2)
  e^{-\frac{\Delta t}{2\varepsilon}}
  \Delta u_k^n w_k^n
  -
  \frac a2
  e^{-\frac{\Delta t}{2\varepsilon}}
  (w_k^n)^2.
\end{aligned}
\]
\end{Lem}

\begin{proof}
We give the main algebraic steps. Set
\[
  A
  :=
  \frac{\Delta u_{j-1}^n+\Delta u_j^n}{2},
  \qquad
  B
  :=
  \frac{\Delta u_j^n-\Delta u_{j-1}^n}{2},
\]
and
\[
  D
  :=
  \frac{w_{j-1}^n+w_j^n}{2},
  \qquad
  C
  :=
  \frac{w_j^n-w_{j-1}^n}{2}.
\]
Subtracting the first relation in \eqref{eq:ubar-half-stag} from the first
relation in \eqref{eq:u-half-compact} gives
\begin{equation}
\label{eq:Du-half-thesis-final}
  \Delta u_{j-\frac12}^{n+\frac12}
  =
  A
  -
  a\sigma B
  -
  \sigma
  e^{-\frac{\Delta t}{2\varepsilon}}
  C.
\end{equation}
Similarly, the first relation in \eqref{eq:w-half-u-w-form} gives
\begin{equation}
\label{eq:w-half-thesis-final}
\begin{aligned}
  \left(
    1+\frac{\Delta t}{2\varepsilon}
  \right)
  w_{j-\frac12}^{n+\frac12}
  =
  D
  -
  (\lambda^2-a^2)\sigma B
  +
  a\sigma
  e^{-\frac{\Delta t}{2\varepsilon}}
  C
  -
  \frac{(\lambda^2-a^2)\Delta t}{2}
  (\tilde\delta_x\bar u)_{j-\frac12}^n.
\end{aligned}
\end{equation}
Using \eqref{eq:Q-split-w-stag}-\eqref{eq:entropy-stagg}, expanding
\[
  \frac{\lambda^2-a^2}{2}
  \left(
    \Delta u_{j-\frac12}^{n+\frac12}
  \right)^2
  +
  \frac12
  \left(
    w_{j-\frac12}^{n+\frac12}
  \right)^2,
\]
and applying Young's inequality to the term involving the discrete derivative
of the reference solution in \eqref{eq:w-half-thesis-final}, we obtain
\begin{align}
\label{eq:intermediate-thesis-final}
  \eta_{j-\frac12,\mathrm{rel}}^{n+\frac12}
  &\leq
  \frac12
  \left(
    \eta_{j-1,\mathrm{rel}}^n
    +
    \eta_{j,\mathrm{rel}}^n
  \right)
  -
  \frac{\sigma}{2}
  \left(
    \mathcal B_{j,\mathrm{rel}}^n
    -
    \mathcal B_{j-1,\mathrm{rel}}^n
  \right)
  \notag\\
  &\quad
  -
  \frac{\Delta t}{4\varepsilon}
  \left(
    w_{j-\frac12}^{n+\frac12}
  \right)^2
  -
  \frac{\lambda^2-a^2}{2}
  (1-\lambda^2\sigma^2)B^2
  \notag\\
  &\quad
  -
  \frac12
  \left(
    1-\lambda^2\sigma^2
    e^{-\frac{\Delta t}{\varepsilon}}
  \right)C^2
  -
  (\lambda^2-a^2)\sigma
  \left(
    1-e^{-\frac{\Delta t}{2\varepsilon}}
  \right)DB
  \notag\\
  &\quad
  +
  \frac{(\lambda^2-a^2)^2\varepsilon}{4}
  \Delta t
  \left(
    (\tilde\delta_x\bar u)_{j-\frac12}^n
  \right)^2.
\end{align}
It remains to estimate the mixed term involving \(D\) and \(B\). Since only
an upper bound is required,
\[
  -
  (\lambda^2-a^2)\sigma
  \left(
    1-e^{-\frac{\Delta t}{2\varepsilon}}
  \right)DB
  \leq
  (\lambda^2-a^2)\sigma
  \left(
    1-e^{-\frac{\Delta t}{2\varepsilon}}
  \right)|DB|.
\]
For every \(\theta>0\), Young's inequality gives
\[
  |DB|
  \leq
  \theta B^2
  +
  \frac{1}{4\theta}D^2.
\]
We choose
\[
  \theta
  =
  \frac{1-\lambda^2\sigma^2}{4\sigma},
\]
which is positive under \eqref{eq:cfl-global-final}. The resulting
\(B^2\)-term combines with the corresponding negative term in
\eqref{eq:intermediate-thesis-final} as follows:
\[
\begin{aligned}
  &-
  \frac{\lambda^2-a^2}{2}
  (1-\lambda^2\sigma^2)B^2
  +
  (\lambda^2-a^2)\sigma
  \left(
    1-e^{-\frac{\Delta t}{2\varepsilon}}
  \right)
  \frac{1-\lambda^2\sigma^2}{4\sigma}
  B^2\\
  &=
  -
  \frac{\lambda^2-a^2}{4}
  (1-\lambda^2\sigma^2)
  \left(
    1+e^{-\frac{\Delta t}{2\varepsilon}}
  \right)B^2
  \leq0.
\end{aligned}
\]
Moreover,
\[
  \frac{1}{4\theta}
  =
  \frac{\sigma}{1-\lambda^2\sigma^2},
\]
and hence the \(D^2\)-term satisfies
\[
  (\lambda^2-a^2)\sigma
  \left(
    1-e^{-\frac{\Delta t}{2\varepsilon}}
  \right)
  \frac{1}{4\theta}D^2
  =
  \frac{
    (\lambda^2-a^2)\sigma^2
  }{
    1-\lambda^2\sigma^2
  }
  \left(
    1-e^{-\frac{\Delta t}{2\varepsilon}}
  \right)D^2.
\]
Using \(1-e^{-x}\leq x\) with
\(x=\frac{\Delta t}{2\varepsilon}\), together with
\[
  D^2
  \leq
  \frac{(w_{j-1}^n)^2+(w_j^n)^2}{2},
\]
we infer that
\[(\lambda^2-a^2)\sigma
  \left(
    1-e^{-\frac{\Delta t}{2\varepsilon}}
  \right)
  \frac{1}{4\theta}D^2
  \leq
  \kappa_\sigma
  \frac{\Delta t}{\varepsilon}
  \frac{(w_{j-1}^n)^2+(w_j^n)^2}{2}.
\]
Substituting these estimates into
\eqref{eq:intermediate-thesis-final}, the remaining \(B^2\)- and \(C^2\)-terms
are nonpositive under \eqref{eq:cfl-global-final}. Discarding these terms
yields \eqref{eq:local-half-ineq-stag-corrected}.
\end{proof}
%-----------------------------------------
\begin{Lem}[Local full-step relative entropy inequality]
\label{lem:local-full-stag-corrected}
Assume the subcharacteristic condition \eqref{eq:subchar-final} and the CFL
condition \eqref{eq:cfl-global-final}. Then, for every \(j\in\mathbb Z\) and
\(n\in\mathbb N\),
\begin{align}
\label{eq:local-full-ineq-stag-corrected}
  \eta_{j,\mathrm{rel}}^{n+1}
  &\leq
  \frac12
  \left(
    \eta_{j-\frac12,\mathrm{rel}}^{n+\frac12}
    +
    \eta_{j+\frac12,\mathrm{rel}}^{n+\frac12}
  \right)
  -
  \frac{\sigma}{2}
  \left(
    \mathcal B_{j+\frac12,\mathrm{rel}}^{n+\frac12}
    -
    \mathcal B_{j-\frac12,\mathrm{rel}}^{n+\frac12}
  \right)
  -
  \frac{\Delta t}{4\varepsilon}
  (w_j^{n+1})^2
  \notag\\
  &\quad
  +
  \kappa_\sigma
  \frac{\Delta t}{\varepsilon}
  \frac{
    (w_{j-\frac12}^{n+\frac12})^2
    +
    (w_{j+\frac12}^{n+\frac12})^2
  }{2}
  +
  \frac{(\lambda^2-a^2)^2\varepsilon}{4}
  \Delta t
  \left(
    (\tilde\delta_x\bar u)_j^{n+\frac12}
  \right)^2,
\end{align}
where
\[
\begin{aligned}
  \mathcal B_{k+\frac12,\mathrm{rel}}^{n+\frac12}:=
  \frac{a(\lambda^2-a^2)}{2}
  \left(
    \Delta u_{k+\frac12}^{n+\frac12}
  \right)^2 +(\lambda^2-a^2)
  e^{-\frac{\Delta t}{2\varepsilon}}
  \Delta u_{k+\frac12}^{n+\frac12}
  w_{k+\frac12}^{n+\frac12}
  -\frac a2
  e^{-\frac{\Delta t}{2\varepsilon}}
  \left(
    w_{k+\frac12}^{n+\frac12}
  \right)^2.
\end{aligned}
\]
\end{Lem}

\begin{proof}
The proof follows the same argument as that of
Lemma~\ref{lem:local-half-stag-corrected}, with
\[
  \left(
    \Delta u_{j-1}^n,
    \Delta u_j^n,
    w_{j-1}^n,
    w_j^n,
    (\tilde\delta_x\bar u)_{j-\frac12}^n
  \right)
\]
replaced by
\[
  \left(
    \Delta u_{j-\frac12}^{n+\frac12},
    \Delta u_{j+\frac12}^{n+\frac12},
    w_{j-\frac12}^{n+\frac12},
    w_{j+\frac12}^{n+\frac12},
    (\tilde\delta_x\bar u)_j^{n+\frac12}
  \right).
\]
More precisely, subtracting the second relation in
\eqref{eq:ubar-half-stag} from the second relation in
\eqref{eq:u-half-compact} gives
\[
  \Delta u_j^{n+1}
  =
  A
  -
  a\sigma B
  -
  \sigma
  e^{-\frac{\Delta t}{2\varepsilon}}
  C,
\]
where \(A\), \(B\), \(C\), and \(D\) denote the corresponding averages and
differences at the intermediate staggered level. Similarly, the second
relation in \eqref{eq:w-half-u-w-form} gives
\[
\begin{aligned}
  \left(
    1+\frac{\Delta t}{2\varepsilon}
  \right)
  w_j^{n+1}
  &=
  D
  -
  (\lambda^2-a^2)\sigma B
  +
  a\sigma
  e^{-\frac{\Delta t}{2\varepsilon}}
  C
  -
  \frac{(\lambda^2-a^2)\Delta t}{2}
  (\tilde\delta_x\bar u)_j^{n+\frac12}.
\end{aligned}
\]
The same expansion and Young inequality argument then yield
\eqref{eq:local-full-ineq-stag-corrected}.
\end{proof}

\begin{proof}[Proof of Theorem~\ref{thm:global-stag-final}]
We first sum the half-step estimate
\eqref{eq:local-half-ineq-stag-corrected} over \(j\in\mathbb Z\) and multiply
the resulting inequality by \(\Delta x\). Under
assumption~\ref{ass:stag-H1}, the discrete flux term telescopes and gives no
boundary contribution. Hence,
\begin{equation}
\label{eq:half-global-final-revised}
\begin{aligned}
  \Phi_{\mathrm{rel}}^{n+\frac12}
  +
  \frac{\Delta t}{4\varepsilon}
  \Delta x
  \sum_{j\in\mathbb Z}
  \left(
    w_{j+\frac12}^{n+\frac12}
  \right)^2 &\leq
  \Phi_{\mathrm{rel}}^n
  +
  \kappa_\sigma
  \frac{\Delta t}{\varepsilon}
  \Delta x
  \sum_{j\in\mathbb Z}
  (w_j^n)^2\\
  &\quad
  +
  \frac{(\lambda^2-a^2)^2\varepsilon}{4}
  \Delta t\,\Delta x
  \sum_{j\in\mathbb Z}
  \left(
    (\tilde\delta_x\bar u)_{j+\frac12}^n
  \right)^2.
\end{aligned}
\end{equation}
Similarly, summing the full-step estimate
\eqref{eq:local-full-ineq-stag-corrected} over \(j\in\mathbb Z\) gives
\begin{equation}
\label{eq:full-global-final-revised}
\begin{aligned}
  \Phi_{\mathrm{rel}}^{n+1}
  +
  \frac{\Delta t}{4\varepsilon}
  \Delta x
  \sum_{j\in\mathbb Z}
  (w_j^{n+1})^2
  &\leq
  \Phi_{\mathrm{rel}}^{n+\frac12}
  +
  \kappa_\sigma
  \frac{\Delta t}{\varepsilon}
  \Delta x
  \sum_{j\in\mathbb Z}
  \left(
    w_{j+\frac12}^{n+\frac12}
  \right)^2\\
  &\quad
  +
  \frac{(\lambda^2-a^2)^2\varepsilon}{4}
  \Delta t\,\Delta x
  \sum_{j\in\mathbb Z}
  \left(
    (\tilde\delta_x\bar u)_j^{n+\frac12}
  \right)^2.
\end{aligned}
\end{equation}
We next compare the discrete derivatives of the reference solution at the two
staggered levels. Using the first relation in
\eqref{eq:ubar-half-stag}, we obtain
\[
  (\tilde\delta_x\bar u)_j^{n+\frac12}
  =
  \frac{1+a\sigma}{2}
  (\tilde\delta_x\bar u)_{j-\frac12}^n
  +
  \frac{1-a\sigma}{2}
  (\tilde\delta_x\bar u)_{j+\frac12}^n.
\]
Since the subcharacteristic condition and
\eqref{eq:cfl-global-final} imply \(|a|\sigma<1\), the two coefficients are
nonnegative and their sum is equal to \(1\). By convexity of
\(x\mapsto x^2\),
\begin{equation}
\label{eq:grad-stability-final-revised}
  \Delta x
  \sum_{j\in\mathbb Z}
  \left(
    (\tilde\delta_x\bar u)_j^{n+\frac12}
  \right)^2
  \leq
  \Delta x
  \sum_{j\in\mathbb Z}
  \left(
    (\tilde\delta_x\bar u)_{j+\frac12}^n
  \right)^2.
\end{equation}
Combining \eqref{eq:half-global-final-revised},
\eqref{eq:full-global-final-revised}, and
\eqref{eq:grad-stability-final-revised}, we obtain
\[
\begin{aligned}
  \Phi_{\mathrm{rel}}^{n+1}
  &+
  \frac{\Delta t}{4\varepsilon}
  \Delta x
  \sum_{j\in\mathbb Z}
  (w_j^{n+1})^2+
  \left(
    \frac14-\kappa_\sigma
  \right)
  \frac{\Delta t}{\varepsilon}
  \Delta x
  \sum_{j\in\mathbb Z}
  \left(
    w_{j+\frac12}^{n+\frac12}
  \right)^2\\
  &\leq
  \Phi_{\mathrm{rel}}^n
  +
  \kappa_\sigma
  \frac{\Delta t}{\varepsilon}
  \Delta x
  \sum_{j\in\mathbb Z}
  (w_j^n)^2
  +
  \frac{(\lambda^2-a^2)^2}{2}
  \varepsilon\Delta t\,\Delta x
  \sum_{j\in\mathbb Z}
  \left(
    (\tilde\delta_x\bar u)_{j+\frac12}^n
  \right)^2.
\end{aligned}
\]
Under \eqref{eq:cfl-global-final},
\[
\begin{aligned}
  \kappa_\sigma
  =
  \frac{
    (\lambda^2-a^2)\sigma^2
  }{
    2(1-\lambda^2\sigma^2)
  }\leq
  \frac{
    \lambda^2\sigma^2
  }{
    2(1-\lambda^2\sigma^2)
  }
  \leq
  \frac14.
\end{aligned}
\]
Therefore, the intermediate dissipation term is nonnegative and may be
discarded. Moreover,
\[
  \frac{\Delta t}{4\varepsilon}
  \Delta x
  \sum_{j\in\mathbb Z}
  (w_j^{n+1})^2
  \geq
  \kappa_\sigma
  \frac{\Delta t}{\varepsilon}
  \Delta x
  \sum_{j\in\mathbb Z}
  (w_j^{n+1})^2.
\]
We thus obtain
\[
\begin{aligned}
  \Phi_{\mathrm{rel}}^{n+1}
  &+
  \kappa_\sigma
  \frac{\Delta t}{\varepsilon}
  \Delta x
  \sum_{j\in\mathbb Z}
  (w_j^{n+1})^2\\
  &\leq
  \Phi_{\mathrm{rel}}^n
  +
  \kappa_\sigma
  \frac{\Delta t}{\varepsilon}
  \Delta x
  \sum_{j\in\mathbb Z}
  (w_j^n)^2
  +
  \frac{(\lambda^2-a^2)^2}{2}
  \varepsilon\Delta t\,\Delta x
  \sum_{j\in\mathbb Z}
  \left(
    (\tilde\delta_x\bar u)_{j+\frac12}^n
  \right)^2.
\end{aligned}
\]
Summing this inequality from \(n=0\) to \(N-1\) and applying
assumption~\ref{ass:stag-H2} yields
\eqref{eq:global-bound-final}.
\end{proof}

Estimate \eqref{eq:global-bound-final} is the discrete counterpart, for the
staggered scheme, of the continuous relative entropy estimate. In particular,
for well-prepared initial data, it yields
\[
  \Phi_{\mathrm{rel}}^N=O(\varepsilon)
\]
uniformly for \(N\Delta t\leq T\).
%============================================================
%==============================================================
% Fixed-time analysis and numerical illustration
%==============================================================
\section{Fixed-Time Analysis and Numerical Illustration}
\label{sec:fixed-time-regimes}

Proposition~\ref{prop:relative-entropy-JX} and
Theorems~\ref{thm:global-disc-rel-entropy} and
\ref{thm:global-stag-final} yield relative entropy estimates of order
\(O(\varepsilon)\) for well-prepared initial data. However,
Figure~\ref{fig:intro-fixed-time} reveals an additional quadratic regime for
the smallest values of the relaxation parameter. The purpose of this section
is to identify the mechanism responsible for this \(O(\varepsilon^2)\)
behavior.

We fix the spatial mesh, the time step, and the final time
\(
  T=t^N=N\Delta t,
\)
and vary only \(\varepsilon\). Consequently, the time index \(N\) remains fixed
throughout the analysis. 

\subsection{Numerical setting and relative entropy decomposition}
\label{subsec:fixed-time-numerical-setting}

The computations are performed on \(\Omega=[-1,1]\) up to the final time
\(T=0.1\), with
\[
  a=2,
  \qquad
  \lambda=4,
  \qquad
  N_x=1000,
  \qquad
  \Delta x=2\times10^{-3},
  \qquad
  \frac{\lambda\Delta t}{\Delta x}\leq0.9.
\]
The relaxation parameter is sampled at \(30\) logarithmically spaced values
between \(10^{-10}\) and \(10^{-1}\). 

We consider the well-prepared Riemann initial data
\[
  u_j^0
  =
  \begin{cases}
    1, & x_j<0,\\
    2, & x_j\geq0,
  \end{cases}
  \qquad
  v_j^0=a u_j^0,
  \qquad
  \bar u_j^0=u_j^0.
\]
Consequently,
\[
  w_j^0=0,
  \qquad
  \Delta u_j^0=0,
\]
where \(w_j^n\) and \(\Delta u_j^n\) are defined in
\eqref{eq:w-def-disc} and \eqref{eq:DeltaU-def-rel}, respectively.

Using \eqref{eq:delta-v-w-du} in the discrete relative entropy functional
\eqref{eq:phi-discret}, we obtain the exact decomposition
\begin{equation}
\label{eq:fixed-time-relative-entropy}
  \Phi_{\mathrm{rel}}^n
  =
  \frac{\lambda^2-a^2}{2}
  \|\Delta u^n\|_{\Delta x}^2
  +
  \frac12
  \|w^n\|_{\Delta x}^2,
  \qquad
  \|z^n\|_{\Delta x}^2
  :=
  \Delta x
  \sum_{j\in\mathbb Z}(z_j^n)^2.
\end{equation}
Hence, if
\[
  \|\Delta u^N\|_{\Delta x}=O(\varepsilon),
  \qquad
  \|w^N\|_{\Delta x}=O(\varepsilon),
\]
then the quadratic structure of the entropy yields
\[
  \Phi_{\mathrm{rel}}^N=O(\varepsilon^2).
\]

\subsection{Lie splitting scheme}
\label{subsec:fixed-time-splitting}

Combining the convective and relaxation steps
\eqref{eq:conv-step}-\eqref{eq:source-step} and subtracting the limiting
scheme \eqref{eq:limit-scheme}, we obtain
\begin{align}
  \Delta u_j^{n+1}
  &=
  \Delta u_j^n
  -
  \frac{a\Delta t}{2\Delta x}
  \left(
    \Delta u_{j+1}^n-\Delta u_{j-1}^n
  \right)
  +
  \frac{\lambda\Delta t}{2\Delta x}
  \left(
    \Delta u_{j+1}^n
    -
    2\Delta u_j^n
    +
    \Delta u_{j-1}^n
  \right)
  \notag\\
  &\quad
  -
  \frac{\Delta t}{2\Delta x}
  \left(
    w_{j+1}^n-w_{j-1}^n
  \right),
  \label{eq:split-fixed-time-du}\\
  w_j^{n+1}
  &=
  \rho w_j^n
  -
  \rho\Delta t\,X_j^n
  +
  R_j^n,
  \qquad
  \rho
  :=
  \frac{\varepsilon}{\varepsilon+\Delta t}.
  \label{eq:split-fixed-time-w}
\end{align}
Here,
\begin{equation}
\label{eq:split-fixed-time-XR}
\begin{aligned}
  X_j^n
  &:=
  (\lambda^2-a^2)
  \frac{u_{j+1}^n-u_{j-1}^n}{2\Delta x},
  \\
  R_j^n
  &:=
  \frac{\rho a\Delta t}{2\Delta x}
  \left(
    w_{j+1}^n-w_{j-1}^n
  \right)
  +
  \frac{\rho\lambda\Delta t}{2\Delta x}
  \left(
    w_{j+1}^n
    -
    2w_j^n
    +
    w_{j-1}^n
  \right).
\end{aligned}
\end{equation}
In particular, every term in \(R_j^n\) contains the factor \(\rho\), and
\[
  \rho\Delta t
  =
  \varepsilon(1-\rho).
\]
For \(m\geq1\), define
\[
  S_m(\varepsilon)
  :=
  \varepsilon(1-\rho^m).
\]
Iterating \eqref{eq:split-fixed-time-w} gives
\begin{equation}
\label{eq:split-w-decomposition}
  w_j^m
  =
  w_{1,j}^m
  -
  S_m(\varepsilon)w_{2,j}^m,
\end{equation}
where
\begin{equation}
\label{eq:split-w12-definitions}
\begin{aligned}
  w_{1,j}^m
  :=
  \rho^m w_j^0
  +
  \sum_{k=0}^{m-1}
  \rho^{m-1-k}R_j^k,
  \qquad
  w_{2,j}^m
  :=
  \frac{1-\rho}{1-\rho^m}
  \sum_{k=0}^{m-1}
  \rho^{m-1-k}X_j^k.
\end{aligned}
\end{equation}
Indeed, the identity
\[
  (1-\rho)
  \sum_{k=0}^{m-1}\rho^{m-1-k}
  =
  1-\rho^m
\]
shows that the contribution generated by \(X_j^k\) factors through
\(S_m(\varepsilon)\).

For a grid function \(z=(z_j)_j\), define
\[
  (\mathcal D_xz)_j
  :=
  \frac{z_{j+1}-z_{j-1}}{2\Delta x},
\]
and introduce the linear transport operator
\[
  (\mathcal Tz)_j
  :=
  z_j
  -
  \frac{a\Delta t}{2\Delta x}
  (z_{j+1}-z_{j-1})
  +
  \frac{\lambda\Delta t}{2\Delta x}
  (z_{j+1}-2z_j+z_{j-1}).
\]
Then \eqref{eq:split-fixed-time-du} can be written as
\[
  \Delta u^{n+1}
  =
  \mathcal T\Delta u^n
  -
  \Delta t\,\mathcal D_xw^n.
\]
Iterating this relation up to the fixed time index \(N\) and using
\eqref{eq:split-w-decomposition}, we obtain
\begin{equation}
\label{eq:split-du-decomposition}
  \Delta u^N
  =
  \Delta u_1^N
  +
  S_N(\varepsilon)\Delta u_2^N,
\end{equation}
where
\begin{equation}
\label{eq:split-du12-definitions}
\begin{aligned}
  \Delta u_1^N
  &:=
  \mathcal T^N\Delta u^0
  -
  \Delta t
  \sum_{k=0}^{N-1}
  \mathcal T^{N-1-k}
  \mathcal D_xw_1^k,
  \\
  \Delta u_2^N
  &:=
  \Delta t
  \sum_{k=1}^{N-1}
  \frac{1-\rho^k}{1-\rho^N}
  \mathcal T^{N-1-k}
  \mathcal D_xw_2^k.
\end{aligned}
\end{equation}
Substituting \eqref{eq:split-w-decomposition} and
\eqref{eq:split-du-decomposition} into
\eqref{eq:fixed-time-relative-entropy} gives
\begin{equation}
\label{eq:split-fixed-time-entropy-decomposition}
\begin{aligned}
  \Phi_{\mathrm{rel}}^N
  =
  \frac{\lambda^2-a^2}{2}
  \left\|
    \Delta u_1^N
    +
    S_N(\varepsilon)\Delta u_2^N
  \right\|_{\Delta x}^2
  +
  \frac12
  \left\|
    w_1^N
    -
    S_N(\varepsilon)w_2^N
  \right\|_{\Delta x}^2.
\end{aligned}
\end{equation}
Since the mesh, the time step, and the time index \(N\) are fixed, the
recurrences above involve only finitely many linear operations whose
coefficients are uniformly bounded with respect to \(\varepsilon\). For the
matching initial data \(w^0=\Delta u^0=0\), and since each remainder
\(R_j^n\) contains the factor \(\rho\), one obtains
\[
  \|w_1^N\|_{\Delta x}
  +
  \|\Delta u_1^N\|_{\Delta x}
  \leq
  C\rho,
  \qquad
  \|w_2^N\|_{\Delta x}
  +
  \|\Delta u_2^N\|_{\Delta x}
  \leq
  C,
\]
where \(C>0\) may depend on the fixed discretization parameters and on \(N\),
but is independent of \(\varepsilon\).

It follows from
\eqref{eq:split-fixed-time-entropy-decomposition} and the inequality
\[
  \|x+y\|^2
  \leq
  2\|x\|^2+2\|y\|^2
\]
that
\begin{equation}
\label{eq:split-rho-S-bound}
  \Phi_{\mathrm{rel}}^N
  \leq
  C
  \left(
    \rho^2+S_N(\varepsilon)^2
  \right).
\end{equation}
In the regime \(\varepsilon\ll\Delta t\),
\[
  0<\rho
  \leq
  \frac{\varepsilon}{\Delta t},
  \qquad
  0\leq
  S_N(\varepsilon)
  \leq
  \varepsilon.
\]
Since \(\Delta t\) is fixed, estimate
\eqref{eq:split-rho-S-bound} yields
\begin{equation}
\label{eq:split-fixed-time-quadratic}
  \Phi_{\mathrm{rel}}^N
  =
  O(\varepsilon^2),
  \qquad
  \varepsilon\ll\Delta t.
\end{equation}
This quadratic behavior is confirmed by the numerical results shown in
Figure~\ref{fig:fixed-time-splitting}.
\begin{figure}[!t]
  \centering
  \includegraphics[width=0.76\textwidth]{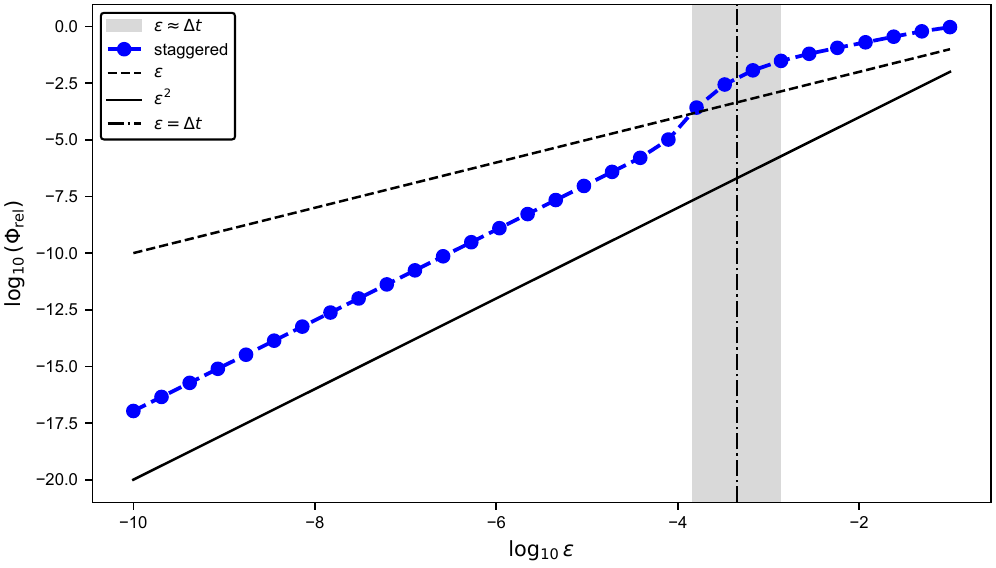}
  \caption{Fixed-time discrete relative entropy for the Lie splitting scheme.
  The vertical line indicates \(\varepsilon=\Delta t\), and the reference
  lines have slopes \(1\) and \(2\)}
  \label{fig:fixed-time-splitting}
\end{figure}

\subsection{Staggered scheme}
\label{subsec:fixed-time-staggered}

For the staggered scheme \eqref{eq:disc-JX-stag}, define
\begin{equation}
\label{eq:alpha-Theta-fixed-time}
  \alpha
  :=
  \frac{2\varepsilon}{2\varepsilon+\Delta t},
  \qquad
  \Theta_m(\varepsilon)
  :=
  \varepsilon
  \frac{1-\alpha^{2m}}{1+\alpha}.
\end{equation}
The identity
\[
  \frac{\alpha\Delta t}{2}
  =
  \varepsilon(1-\alpha),
\]
together with the two relations in
\eqref{eq:w-half-u-w-form}, yields the one-step recurrence
\begin{equation}
\label{eq:staggered-w-one-step}
  w_j^{n+1}
  =
  \alpha^2w_j^n
  -
  \alpha^2\frac{\Delta t}{2}X_{j-\frac12}^n
  -
  \alpha\frac{\Delta t}{2}Y_j^n
  +
  \widetilde R_j^{n+1},
\end{equation}
where
\begin{equation}
\label{eq:staggered-XY}
  X_{j-\frac12}^n
  :=
  (\lambda^2-a^2)
  \frac{u_j^n-u_{j-1}^n}{\Delta x},
  \qquad
  Y_j^n
  :=
  (\lambda^2-a^2)
  \frac{
    u_{j+\frac12}^{n+\frac12}
    -
    u_{j-\frac12}^{n+\frac12}
  }{\Delta x},
\end{equation}
and
\begin{equation}
\label{eq:staggered-remainder}
\begin{aligned}
  \widetilde R_j^{n+1}
  &=
  \alpha^2
  \left(
    -\frac12
    +
    \frac{a\Delta t}{2\Delta x}
    e^{-\frac{\Delta t}{2\varepsilon}}
  \right)
  (w_j^n-w_{j-1}^n)
  \\
  &\quad
  +
  \alpha
  \left(
    \frac12
    +
    \frac{a\Delta t}{2\Delta x}
    e^{-\frac{\Delta t}{2\varepsilon}}
  \right)
  \left(
    w_{j+\frac12}^{n+\frac12}
    -
    w_{j-\frac12}^{n+\frac12}
  \right).
\end{aligned}
\end{equation}
Thus, every term in \(\widetilde R_j^{n+1}\) contains at least one factor
\(\alpha\).

Iterating \eqref{eq:staggered-w-one-step} gives, for \(m\geq1\),
\begin{equation}
\label{eq:staggered-w-decomposition}
  w_j^m
  =
  w_{1,j}^m
  -
  \Theta_m(\varepsilon)w_{2,j}^m,
\end{equation}
where
\begin{equation}
\label{eq:staggered-w12-definitions}
\begin{aligned}
  w_{1,j}^m
  &:=
  \alpha^{2m}w_j^0
  +
  \sum_{k=0}^{m-1}
  \alpha^{2(m-1-k)}
  \widetilde R_j^{k+1},
  \\
  w_{2,j}^m
  &:=
  \frac{1-\alpha^2}{1-\alpha^{2m}}
  \sum_{k=0}^{m-1}
  \alpha^{2(m-1-k)}
  \left(
    \alpha X_{j-\frac12}^k+Y_j^k
  \right).
\end{aligned}
\end{equation}
Indeed,
\[
  \varepsilon(1-\alpha)
  \sum_{k=0}^{m-1}
  \alpha^{2(m-1-k)}
  =
  \Theta_m(\varepsilon)
  \frac{1-\alpha^2}{1-\alpha^{2m}}
  \sum_{k=0}^{m-1}
  \alpha^{2(m-1-k)}.
\]

Subtracting the two limiting updates in
\eqref{eq:ubar-half-stag} from the corresponding two relaxation updates in
\eqref{eq:u-half-compact}, and inserting
\eqref{eq:staggered-w-decomposition}, yields after iteration up to \(N\)
\begin{equation}
\label{eq:staggered-du-decomposition}
  \Delta u^N
  =
  \Delta u_1^N
  +
  e^{-\frac{\Delta t}{2\varepsilon}}
  \Theta_N(\varepsilon)\Delta u_2^N.
\end{equation}
Here, \(\Delta u_1^N\) collects the contributions generated by the initial
error and by the remainder terms, whereas \(\Delta u_2^N\) contains the
normalized forcing contributions associated with \(X\) and \(Y\). The
coefficients occurring in the latter are uniformly bounded because
\[
  0
  \leq
  \frac{\Theta_m(\varepsilon)}{\Theta_N(\varepsilon)}
  =
  \frac{1-\alpha^{2m}}{1-\alpha^{2N}}
  \leq
  1,
  \qquad
  0\leq m\leq N.
\]

Combining \eqref{eq:staggered-w-decomposition},
\eqref{eq:staggered-du-decomposition}, and
\eqref{eq:fixed-time-relative-entropy}, we obtain
\begin{equation}
\label{eq:staggered-fixed-time-entropy-decomposition}
\begin{aligned}
  \Phi_{\mathrm{rel}}^N
  &=
  \frac{\lambda^2-a^2}{2}
  \left\|
    \Delta u_1^N
    +
    e^{-\frac{\Delta t}{2\varepsilon}}
    \Theta_N(\varepsilon)\Delta u_2^N
  \right\|_{\Delta x}^2
  \\
  &\quad
  +
  \frac12
  \left\|
    w_1^N
    -
    \Theta_N(\varepsilon)w_2^N
  \right\|_{\Delta x}^2.
\end{aligned}
\end{equation}

For the matching initial data and fixed discretization parameters,
\eqref{eq:staggered-remainder} and
\eqref{eq:staggered-w12-definitions} imply
\[
  \|w_1^N\|_{\Delta x}
  +
  \|\Delta u_1^N\|_{\Delta x}
  \leq
  C\alpha,
  \qquad
  \|w_2^N\|_{\Delta x}
  +
  \|\Delta u_2^N\|_{\Delta x}
  \leq
  C,
\]
where \(C>0\) is independent of \(\varepsilon\). Since
\[
  0<
  e^{-\frac{\Delta t}{2\varepsilon}}
  \leq1,
\]
we deduce from
\eqref{eq:staggered-fixed-time-entropy-decomposition} that
\begin{equation}
\label{eq:staggered-alpha-Theta-bound}
  \Phi_{\mathrm{rel}}^N
  \leq
  C
  \left(
    \alpha^2+\Theta_N(\varepsilon)^2
  \right).
\end{equation}
In the regime \(\varepsilon\ll\Delta t\),
\[
  0<\alpha
  \leq
  \frac{2\varepsilon}{\Delta t},
  \qquad
  0\leq
  \Theta_N(\varepsilon)
  \leq
  \varepsilon.
\]
Since \(\Delta t\) is fixed, it follows that
\begin{equation}
\label{eq:staggered-fixed-time-quadratic}
  \Phi_{\mathrm{rel}}^N
  =
  O(\varepsilon^2),
  \qquad
  \varepsilon\ll\Delta t.
\end{equation}
The corresponding quadratic regime for the staggered scheme is illustrated in
Figure~\ref{fig:fixed-time-staggered}.
\begin{figure}[!t]
  \centering
  \includegraphics[width=0.76\textwidth]{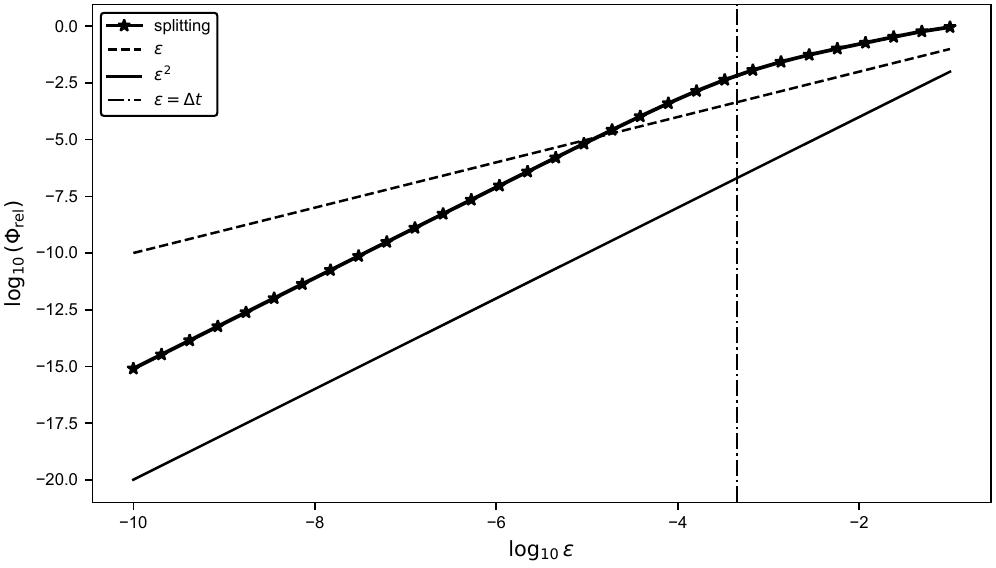}
  \caption{Fixed-time discrete relative entropy for the staggered scheme. The
  vertical line indicates \(\varepsilon=\Delta t\), the shaded region
  corresponds to comparable relaxation and discretization scales, and the
  reference lines have slopes \(1\) and \(2\)}
  \label{fig:fixed-time-staggered}
\end{figure}

\subsection{Interpretation of the numerical results}
\label{subsec:fixed-time-regime-comparison}

In the regime \(\varepsilon\ll\Delta t\), the factors
\[
  \rho,
  \qquad
  S_N(\varepsilon),
  \qquad
  \alpha,
  \qquad
  \Theta_N(\varepsilon)
\]
are of order \(O(\varepsilon)\) for fixed \(\Delta t\). Consequently,
\[
  \|\Delta u^N\|_{\Delta x}
  +
  \|w^N\|_{\Delta x}
  =
  O(\varepsilon)
\]
for both schemes. The quadratic identity
\eqref{eq:fixed-time-relative-entropy} then yields
\(
  \Phi_{\mathrm{rel}}^N
  =
  O(\varepsilon^2).
\)

When \(\varepsilon\) and \(\Delta t\) are comparable, neither \(\rho\) nor
\(\alpha\) is small. The additional damping mechanism responsible for the
quadratic regime is therefore progressively lost. The transition is sharper
for the Lie splitting scheme, in which the transport and relaxation operators
are applied successively. In the staggered method, the relaxation mechanism
also enters through the intermediate states, consistently with the smoother
transition observed in the numerical results.

Finally, when \(\Delta t\ll\varepsilon\), the relaxation scale is resolved by
the temporal discretization. In this regime, the fixed-time decompositions no
longer provide the additional small factors used in the quadratic estimates.
The global \(O(\varepsilon)\) relative entropy estimates established in
Theorems~\ref{thm:global-disc-rel-entropy} and
\ref{thm:global-stag-final} therefore remain the relevant uniform bounds.
%-------------------------------------------
\section{Conclusion}

We have established continuous and fully discrete relative entropy estimates
for the linear Jin-Xin relaxation system. For well-prepared initial data, the
continuous problem and the two asymptotic-preserving discretizations considered
in this work satisfy uniform first-order estimates with respect to the
relaxation parameter.

The numerical results additionally exhibit a quadratic regime when the
relaxation scale is smaller than the time step. The fixed-time analysis
identifies the mechanism responsible for this sharper behavior and explains
the different transitions observed for the Lie splitting and staggered
schemes.

\medskip
\textbf{Acknowledgements.}
This work received financial support from the CNRS grant
\textit{D\'efi Math\'ematiques France 2030} and from the Inria project-team
\textit{ANGUS}.

\bibliographystyle{plain}
\bibliography{biblio}

\begin{thebibliography}{10}

\bibitem{BerthonBessemoulinMathis16}
C.~Berthon, M.~Bessemoulin-Chatard, and H.~Mathis.
\newblock Numerical convergence rate for a diffusive limit of hyperbolic
  systems: {$p$}-system with damping.
\newblock {\em SMAI Journal of Computational Mathematics}, 2:99--119, 2016.

\bibitem{BessemoulinMathis24}
M.~Bessemoulin-Chatard and H.~Mathis.
\newblock Relative entropy for the numerical diffusive limit of the linear
  {Jin-Xin} system.
\newblock {\em ESAIM: Proceedings and Surveys}, 79:126--138, 2025.

\bibitem{BouchutBook04}
F.~Bouchut.
\newblock {\em Nonlinear Stability of Finite Volume Methods for Hyperbolic
  Conservation Laws and Well-Balanced Schemes for Sources}.
\newblock Frontiers in Mathematics. Birkh{\"a}user, Basel, 2004.

\bibitem{BulteauBerthonBessemoulin19}
S.~Bulteau, C.~Berthon, and M.~Bessemoulin-Chatard.
\newblock Convergence rate of an asymptotic-preserving scheme for the diffusive
  limit of the {$p$}-system with damping.
\newblock {\em Communications in Mathematical Sciences}, 17(6):1459--1486,
  2019.

\bibitem{ChenLevermoreLiu94}
G.-Q. Chen, C.~D. Levermore, and T.-P. Liu.
\newblock Hyperbolic conservation laws with stiff relaxation terms and entropy.
\newblock {\em Communications on Pure and Applied Mathematics}, 47(6):787--830,
  1994.

\bibitem{DafermosBook10}
C.~M. Dafermos.
\newblock {\em Hyperbolic Conservation Laws in Continuum Physics}, volume 325
  of {\em Grundlehren der Mathematischen Wissenschaften}.
\newblock Springer, Berlin, 3 edition, 2010.

\bibitem{HartenLaxvanLeer1983}
A.~Harten, P.~D. Lax, and B.~van Leer.
\newblock On upstream differencing and godunov-type schemes for hyperbolic
  conservation laws.
\newblock {\em SIAM Review}, 25(1):35--61, 1983.

\bibitem{jin2010asymptotic}
S.~Jin.
\newblock Asymptotic preserving ({AP}) schemes for multiscale kinetic and
  hyperbolic equations: a review.
\newblock {\em Rivista di Matematica della Universit{\`a} di Parma},
  3(2):177--216, 2012.

\bibitem{Jin2022}
S.~Jin.
\newblock Asymptotic-preserving schemes for multiscale physical problems.
\newblock {\em Acta Numerica}, 31:415--489, 2022.

\bibitem{JinLevermore96}
S.~Jin and C.~D. Levermore.
\newblock Numerical schemes for hyperbolic conservation laws with stiff
  relaxation terms.
\newblock {\em Journal of Computational Physics}, 126(2):449--467, 1996.

\bibitem{JinXin1995}
S.~Jin and Z.~Xin.
\newblock The relaxation schemes for systems of conservation laws in arbitrary
  space dimensions.
\newblock {\em Communications on Pure and Applied Mathematics}, 48(3):235--276,
  1995.

\bibitem{LattanzioTzavaras12}
C.~Lattanzio and A.~E. Tzavaras.
\newblock Relative entropy in diffusive relaxation.
\newblock {\em SIAM Journal on Mathematical Analysis}, 45(3):1563--1584, 2013.

\bibitem{MahmoudMathis2025}
C.~Mahmoud and H.~Mathis.
\newblock Asymptotic preserving schemes for hyperbolic systems with relaxation,
  2025.
\newblock Preprint, \href{https://hal.science/hal-05291431v1}{HAL:
  hal-05291431v1}.

\bibitem{Natalini1996}
R.~Natalini.
\newblock Convergence to equilibrium for the relaxation approximations of
  conservation laws.
\newblock {\em Communications on Pure and Applied Mathematics}, 49(8):795--823,
  1996.

\bibitem{SerreBook99}
D.~Serre.
\newblock {\em Systems of Conservation Laws 1: Hyperbolicity, Entropies, Shock
  Waves}.
\newblock Cambridge University Press, Cambridge, 1999.

\bibitem{SerreRelax2000}
D.~Serre.
\newblock Relaxations semi-lin{\'e}aire et cin{\'e}tique des syst{\`e}mes de
  lois de conservation.
\newblock {\em Annales de l'Institut Henri Poincar{\'e} C, Analyse non
  lin{\'e}aire}, 17(2):169--192, 2000.

\bibitem{Tadmor2003}
E.~Tadmor.
\newblock Entropy stability theory for difference approximations of nonlinear
  conservation laws and related time-dependent problems.
\newblock {\em Acta Numerica}, 12:451--512, 2003.

\bibitem{Tzavaras05}
A.~E. Tzavaras.
\newblock Relative entropy in hyperbolic relaxation.
\newblock {\em Communications in Mathematical Sciences}, 3(2):119--132, 2005.

\end{thebibliography}

\end{document}